\documentclass[12pt]{amsart}
\usepackage{amssymb}
\usepackage[mathscr]{eucal}

\usepackage{colordvi}
\usepackage{graphicx}
\usepackage[ps2pdf]{hyperref}
\usepackage{url}
\usepackage{vmargin}
\setpapersize{USletter}
\setmargrb{2.5cm}{2.5cm}{2.5cm}{2.5cm} 

\numberwithin{equation}{section}
\newtheorem{itheorem}{Theorem}
\newtheorem{theorem}{Theorem}[section]
\newtheorem{proposition}[theorem]{Proposition}
\newtheorem{lemma}[theorem]{Lemma}

\theoremstyle{remark}

\newtheorem{example}[theorem]{Example}
\newtheorem{remark}[theorem]{Remark}
\newtheorem{definition}[theorem]{Definition}

\newcounter{FNC}[page]
\def\fauxfootnote#1{{\addtocounter{FNC}{2}\Magenta{$^\fnsymbol{FNC}$}%
     \let\thefootnote\relax\footnotetext{\Magenta{$^\fnsymbol{FNC}$#1}}}}
\newcommand{\QED}{\includegraphics[height=10pt]{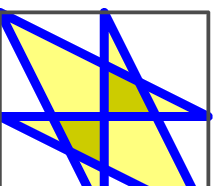}}

\newcommand{\C}{{\bf{C}}}
\renewcommand{\P}{{\bf{P}}}
\newcommand{\R}{{\bf{R}}}
\newcommand{\T}{{\bf{T}}}
\newcommand{\Z}{{\bf{Z}}}

\newcommand{\calA}{{\mathcal{A}}}
\newcommand{\calP}{{\mathcal{P}}}

\newcommand{\tp}{\widetilde{p}\hspace{1.2pt}}
\newcommand{\tpp}{{\widetilde{p}\hspace{1.6pt}'\hspace{-3.4pt}}}
\newcommand{\tq}{\widetilde{q}}
\newcommand{\tqp}{{\widetilde{q}\hspace{1.6pt}'\hspace{-3.4pt}}}
\newcommand{\ba}{{\bf a}}
\newcommand{\bb}{{\bf b}}
\newcommand{\bc}{{\bf c}}
\newcommand{\bd}{{\bf d}}
\newcommand{\be}{{\bf e}}
\newcommand{\bbf}{{\bf f}}
\newcommand{\bg}{{\bf g}}
\newcommand{\bh}{{\bf h}}
\newcommand{\bu}{{\bf u}}
\newcommand{\bv}{{\bf v}}
\newcommand{\bw}{{\bf w}}
\newcommand{\bO}{{\bf 0}}
\newcommand{\pr}{{\it pr}}

\DeclareMathOperator{\Arg}{Arg}
\DeclareMathOperator{\sgn}{sgn}
\DeclareMathOperator{\sign}{sign}
\DeclareMathOperator{\cone}{cone}
\DeclareMathOperator{\conv}{conv}
\DeclareMathOperator{\vol}{vol}

\newcommand{\defcolor}[1]{\Blue{#1}}
\newcommand{\demph}[1]{\defcolor{{\sl #1}}}

\title{Discriminant coamoebas through homology}
\author[Passare]{Mikael Passare$^\dagger$}
\address{Mikael Passare\\
         Department of Mathematics\\
         Stockholm University\\
         SE-106 91 Stockholm\\
         Sweden}
\urladdr{\url{http://www.math.su.se/~passare/}}
\author[Sottile]{Frank Sottile}
\address{Frank Sottile \\
         Department of Mathematics\\
         Texas A\&M University\\
         College Station\\
         Texas \ 77843\\
         USA}
\email{sottile@math.tamu.edu}
\urladdr{\url{http://www.math.tamu.edu/~sottile}}
\thanks{Research of Sottile supported in part by NSF grant DMS-1001615 and the Institut
  Mittag-Leffler} 
\begin{document}

\begin{abstract}
 Understanding the complement of the coamoeba of a (reduced) $A$-discriminant is one approach
 to studying the monodromy of solutions to the corresponding system of $A$-hypergeometric
 differential equations. 
 Nilsson and Passare described the structure of the coamoeba and its complement (a
 zonotope) when the reduced $A$-discriminant is a function of two variables.
 Their main result was that the coamoeba and zonotope form a cycle which is equal to
 the fundamental cycle of the torus, multiplied by the normalized volume of the set $A$ of
 integer vectors.
 That proof only worked in dimension two.
 Here, we use simple ideas from topology to give a new proof of this result in dimension two,
 one which can be generalized to all dimensions.
\end{abstract}

\maketitle

%
\section*{Introduction}

 $A$-hypergeometric functions, which are solutions to $A$-hypergeometric systems of
 differential equations~\cite{GKZ89,GKZ94,SST}, enjoy two complimentary analytical formulae
 which together give an approach to studying the monodromy of the
 solutions~\cite{Beukers} at non-resonant parameters.
 One formula is as explicit power series whose convergence domains in
 $\C^{N+1}$ have an action of the group $\T^{N+1}$ of phases.
 These power series form a basis of solutions, with known local monodromy around loops from 
 $\T^{N+1}$.  
 Another formula is as $A$-hypergeometric Mellin-Barnes
 integrals~\cite{Nilsson} evaluated at phases $\theta\in\T^{N+1}$.
 When the Mellin-Barnes integrals give a basis of solutions, they may be
 used to glue together the local monodromy groups and determine a subgroup
 of the monodromy group, which may sometimes be the full monodromy group.

 Here, $A\subset\Z^n$ consists of $N{+}1$ integer vectors that generate $\Z^n$.
 Considering $\Z^n\subset\Z^{n+1}$ as the vectors with first coordinate 1, we regard $A$ as a
 collection of $N{+}1$ vectors in $\Z^{n+1}$.
 The $A$-discriminant is a multihomogeneous polynomial in $N{+}1$ variables with $n{+}1$
 homogeneities corresponding to $A$.
 Removing these homogeneities gives the reduced $A$-discriminant, $D_B$, which is a hypersurface in 
 $\C^d$ ($d:=N{-}n$) that depends upon a vector configuration $B\subset\Z^d$ Gale dual to
 $A$. 
 This reduction corresponds to a homomorphism $\beta\colon(\C^*)^{N+1}\to(\C^*)^d$ and 
 induces a corresponding map $\Arg(\beta)$ on phases.

 The Mellin-Barnes integrals at $\theta\in\T^{N+1}$ give a basis of solutions when
 $\Arg(\beta)(\theta)$ has a neighborhood in $\T^{d}$ with the property that no point of
 $D_B$ has a phase lying in that neighborhood~\cite{Nilsson}.
 By results in~\cite{Johansson,NS}, this means that $\Arg(\beta)(\theta)$ lies in the
 complement of the closure of the coamoeba $\calA_B$ of $D_B$.

 When $d=2$, the closure of $\calA_B$ and its complement were described
 in~\cite{NP10} as topological chains in $\T^2$ (induced from natural chains in its
 universal cover $\R^2$, where $\T^2=(\R/2\pi\Z)^2$).
 The closure of the coamoeba is an explicit chain depending on 
 $B$.
 Its edges coincide with the edges of the zonotope $Z_B$ generated by $B$.
%
%
 The main result of~\cite{NP10} is the following theorem.

\begin{itheorem}\label{T:NP}
 The sum of the coamoeba chain $\overline{\calA_B}$ and the zonotope $Z_B$ forms a two-dimensional 
 cycle in $\T^2$ that is equal to $n!\vol(A)$ times the fundamental cycle.
\end{itheorem}

 Here, $n!\vol(A)$ is the normalized volume of the convex hull of $A$, which is the dimension
 of the space of solutions to the (non-resonant) $A$-hypergeometric system.
 The zonotope $Z_B$ gives points in the complement of $\calA_B$, by Theorem~\ref{T:NP}.
 Its proof in~\cite{NP10} only works when $d=2$ and it is not clear
 how to generalize it to $d>2$. 
 However, any such generalization would be important, for Mellin-Barnes
 integrals at a set of phases $\theta$ where $\Arg(\beta)(\theta)$ are distinct points of
 $Z_B$ with the same image in $\T^d$ are linearly independent.

 We give a proof of Theorem~\ref{T:NP} which explains the occurrence of the
 zonotope and can be generalized to higher dimensions.
 This proof uses the Horn-Kapranov parametrization of the $A$-discriminant~\cite{K91},
 which implies that the discriminant coamoeba is the image of the coamoeba of a line $\ell_B$ in
 $\P^N$ under the map $\Arg(\beta)$.
 We construct a piecewise linear \demph{zonotope chain} in $\T^N$ (the quotient of
 $\T^{N+1}$ by the diagonal torus) which is a cone over the
 boundary of the coamoeba of $\ell_B$, and compute the homology class of the sum of the coamoeba 
 and this zonotope chain. 
 This gives a formula for the image of this cycle under $\Arg(\beta)$, which 
 we show is $n!\vol(A)$ times the fundamental cycle of $\T^2$.
 Theorem~\ref{T:NP} follows as the map $\Arg(\beta)$ sends the coamoeba of $\ell$ to the
 coamoeba $\calA_B$ of $D_B$ and sends the zonotope chain to $Z_B$.

 While for $A$-discriminants, the set $A$ consists of distinct integer vectors and
 consequently its Gale dual $B$  generates $\Z^2$ and has no  
 two vectors parallel, we establish Theorem~\ref{T:NP} 
 in the greater generality of any finite multiset $B$ of
 integer vectors in $\Z^2$ with sum $\bO$ that spans $\R^2$.
 This generality is useful in our primary application to hypergeometric systems,
 for example the classical systems of Appell~\cite{appell} and Lauricella~\cite{lauricella}
 may be expressed as $A$-hypergeometric systems with repeated vectors in the Gale dual $B$.
 In this setting, we replace the reduced $A$-discriminant by the Horn-Kapranov
 parametrization given by the vectors $B$, and study the coamoeba $\calA_B$ of the
 image, which is also written $D_B$. 
 The normalized volume $n!\vol(A)$ of the configuration $A$ is replaced by a quantity $d_B$
 that depends upon the vectors in $B$.

 We collect some preliminaries in Section~\ref{S:one}.
 In Section~\ref{S:realline} we study the coamoeba of a line in $\P^N$ defined over the real
 numbers and define its associated zonotope chain.
 Our main result is a computation of the homology class of the cycle formed by these two
 chains. 
 In Section~\ref{S:three} we show that under the map $\Arg(\beta)$ the coamoeba and zonotope
 chains map to the coamoeba $\calA_B$ and the zonotope $Z_B$,
 and a simple application of the result in Section~\ref{S:realline} shows that the homology
 class of $\overline{\calA_B}+Z_B$ is $d_B$ times the fundamental cycle of $\T^2$.\medskip

\noindent{\bf Remark.}
 This approach to reduced $A$-discriminant coamoebas and their complements was developed during the
 Winter 2011 semester at the Institut Mittag-Leffler, with the main result obtained in August 
 2011, along with a sketch of a program to extend it to $d\geq 2$.
 With the tragic death of Mikael Passare on 15 September 2011, the task of completing this
 paper fell to the second author, and the program extending these results is being carried out
 in collaboration with Mounir Nisse.

%
\section{Coamoebas and cohomology of tori}\label{S:one}

Throughout $N$ will be an integer strictly greater than 1.
Let $\defcolor{\P^N}$ be $N$-dimensional complex projective space, which will
always have a preferred set of coordinates  $[x_1:\dotsb:x_N:x_{N+1}]$ (up to reordering).
Similarly, $\defcolor{\C^N}$, $\defcolor{(\C^*)^N}$, $\defcolor{\R^N}$, and $\defcolor{\Z^N}$
are $N$-tuples of complex numbers, non-zero complex numbers, real numbers, and integers, all
with corresponding preferred coordinates.
We will write $\be_i$ for the $i$th basis vector in a corresponding ordered basis.

The argument map $\C^*\ni z=re^{\sqrt{-1}\theta}\mapsto\theta\in\T:=\R/2\pi\Z$ induces an
argument map $\defcolor{\Arg}\colon(\C^*)^N\to\T^N$.
To a subvariety $X\subset\P^N$ (or $\C^N$ or $(\C^*)^N$) we associate its \demph{coamoeba}
$\defcolor{\calA(X)}\subset\T^N$ which is the image of $X\cap(\C^*)^N$ under $\Arg$.
The closure of the coamoeba $\calA(X)$ was studied in~\cite{Johansson,NS}.
This closure contains  $\calA(X)$, together with all limits of arguments of unbounded
sequences in $X\cap(\C^*)^N$, which constitute the 
\demph{phase limit set of $X$, $\calP^\infty(X)$}. 
The main result of~\cite{NS} (proven when $X$ is a complete intersection in~\cite{Johansson})
is that  $\calP^\infty(X)$ is the union of the coamoebas of all initial degenerations of
$X\cap(\C^*)^N$. 

Lines in $\C^3$ were studied in~\cite{NS}, and the arguments there imply some basic facts about
coamoebas of lines.
When $X=\defcolor{\ell}\subset\C^N$ is a line which is not parallel to a sum of coordinate
directions 
($\be_{i_1}+\dotsb+\be_{i_s}$ for some subset $\{i_1,\dotsc,i_s\}$ of $\{1,\dotsc,N\}$), its
coamoeba is two-dimensional and its phase limit set is a union of at most $N{+}1$
one-dimensional subtori of $\T^N$, one for each point of $\ell$ at infinity, whose directions are
parallel to sums of coordinate directions.
If $\ell'\subset\C^M$ ($M<N$) is the image of $\ell$ under a coordinate projection, then the
coamoeba $\calA(\ell')$ is the image of $\calA(\ell)$ under the induced projection.
If $\ell'$ is not parallel to a sum of coordinate directions, then the map
$\overline{\calA(\ell)}\to\overline{\calA(\ell')}$ is an injection except for those components
of the phase limit set which are collapsed to points. \medskip

The integral cohomology of the compact torus $\T^N$ is the exterior algebra $\wedge^*\Z^N$.
Under the natural identification of homology with the linear dual of cohomology
(which is again  $\wedge^*\Z^N$), we will write $\be_i$ for the fundamental $1$-cycle
$[\T_i]$ of the coordinate circle $\T_i:= 0^{i-1}\times \T\times 0^{N-i}$ and $\be_i\wedge \be_j$
is the fundamental cycle $[\T_{i,j}]$ of the coordinate 2-torus  $\T_{i,j}\simeq\T^2$ in the
directions $i$ and $j$ with the implied orientation. 
Given a continuous map $\rho\colon\T^N\to\T^2$, the induced map in homology is 
$\rho_*\colon H_*(\T^N,\Z)\to H_*(\T^2,\Z)$ where $\rho_*(\be_i)=[\rho(\T_i)]$, where 
we interpret $[\rho(\T_i)]$ as a cycle---the set of points in 
$\rho(\T_i)$ over which $\rho$ has degree $n$ will appear in $[\rho(\T_i)]$ with coefficient
$n$.
By the identification of $H_*(\T^N,\Z)$ with $\wedge^*\Z^N$,
such a map is determined by its action on $H_1(\T^N,\Z)$, where it is an integer linear map
$\Z^N\to\Z^2$.

%
\section{The coamoeba and zonotope chains of a real line}\label{S:realline}

We study the coamoeba $\calA(\ell)$ of a line $\ell$ in $\P^N$ defined by real
equations.
Its closure $\overline{\calA(\ell)}$ is a two-dimensional chain in $\T^N$ whose boundary consists
of at most $N{+}1$ one-dimensional subtori parallel to sums of coordinate directions.
We describe a piecewise linear two-dimensional chain---the \demph{zonotope chain} of
$\ell$---which has the same boundary as the coamoeba, but with opposite orientation.
The union of the coamoeba and the zonotope chain forms a cycle whose 
homology class we compute.

The line $\ell$ has a parametrization
\[
   \Phi\ \colon\ \P^1\ni z\ \longmapsto\ 
    [b_1(z)\,:\, b_2(z)\,:\, \dotsb\,:\,b_{N+1}(z)]\ \in\ \P^N\,,
\]
where $b_1,\dotsc,b_{N+1}$ are real linear forms with zeroes
$\xi_1,\dotsc,\xi_{N+1}\in\R\P^1$.
The formulation and statement of our results about the coamoeba of $\ell$ will be with
respect to particular orderings of the forms $b_i$, which we now describe.

\begin{definition}\label{D:conventions}
Suppose that these zeroes are in a weakly increasing cyclic order on $\R\P^1$,
 \begin{equation}\label{Eq:A}
   \xi_1\ \leq\ \xi_2\ \leq\ \dotsb\ \leq\ \xi_{N+1}\,.
 \end{equation}
Next, identify $\P^1\smallsetminus\{\xi_{N+1}\}$ with
$\C$, so that $\xi_{N+1}$ is the point $\infty$ at infinity, and suppose that the
distinct zeroes are
 \begin{equation}\label{Eq:B}
   \zeta_1\ <\ \zeta_2\ <\ \dotsb\ <\ \zeta_M\ <\ \zeta_{M{+}1}\ =\ \infty\,.
 \end{equation}
(Note that $M\leq N$.)
Let $\R=\R\P^1\smallsetminus\{\infty\}$ and consider the forms $b_i$ as affine functions on
$\R$.
Fix a scaling of these functions so that $b_{N+1}=1$.
On the interval $(-\infty,\zeta_1)$ the sign of each function $b_i$ is constant.
Define $\defcolor{\sgn_i}\in\{\pm 1\}$ to be this sign.

By~\eqref{Eq:A} and~\eqref{Eq:B}, there exist numbers
$1=m_1<\dotsb<m_{M+1}<m_{M+2}=N{+}2$ such that $b_i(\zeta_j)=0$ if and only if 
$i\in[m_j,m_{j+1})$.
We further suppose that on each of these intervals $[m_j,m_{j+1})$ the signs $\sgn_i$ 
are weakly ordered.
Specifically, there are integers $n_1,\dotsc,n_{M{+}1}$ with $m_j< n_j \leq m_{j+1}$ such that
one of the following holds
 \begin{eqnarray}
   \sgn_{m_j}=\sgn_{m_j+1}=\dotsb=\sgn_{n_j-1}=-1 &<
   &1=\sgn_{n_j}=\dotsb=\sgn_{m_{j+1}-1}\,,\makebox[.1in][l]{\qquad or}\label{Eq:inc}\\
      \sgn_{m_j}=\sgn_{m_j+1}=\dotsb=\sgn_{n_j-1}=1 &>
   &-1=\sgn_{n_j}=\dotsb=\sgn_{m_{j+1}-1}\,,\label{Eq:dec}
 \end{eqnarray}
for $j=1,\dotsc,M{+}1$.
If $n_j=m_{j+1}$, then all the signs are the same; otherwise both
signs occur.
Since $b_{N+1}=1$, either~\eqref{Eq:inc} occurs with $n_{M+1}\leq N{+}1$
or~\eqref{Eq:dec} occurs with $n_{M+1}=N{+}1$.
\hfill\QED
\end{definition}

The point $\Arg(b_1(z),\dotsc,b_N(z))\in\T^N$ is constant for $z$ in
each interval of $\R^1\smallsetminus\{\zeta_1,\dotsc,\zeta_M\}$.
Let $\defcolor{p_1}:=(\arg(\sgn_i)\mid i=1,\dotsc,N)$ be the point coming from the
interval $(-\infty,\zeta_1)$, and for each $j=1,\dotsc,M$, let \defcolor{$p_{j+1}$} be the
point coming from the interval $(\zeta_j,\zeta_{j+1})$.
These $M{+}1$ points $p_1,\dotsc,p_{M+1}$ of $\T^N$ are the vertices of the coamoeba $\calA(\ell)$
of $\ell$. 

To understand the rest of the coamoeba, note that when $M\geq 2$ the map $\Arg\circ\Phi$ is
injective on $\P^1\smallsetminus\R\P^1$ (see~\cite[\S~2]{NS}).
(When $M=1$, $\ell$ is parallel to a sum of coordinate directions and $\calA(\ell)$ is a
translate of the corresponding one-dimensional subtorus of $\T^N$.)
It suffices to consider the image of the upper half plane, as the image of the lower half
plane is obtained by multiplying by $-1$ (induced by complex conjugation).
For the upper half plane, consider $\Arg\circ\Phi(z)$ for $z$ lying on a contour $C$ 
as shown in Figure~\ref{F:contour} 
\begin{figure}[htb]
  \begin{picture}(223,77)
    \put(0,10){\includegraphics{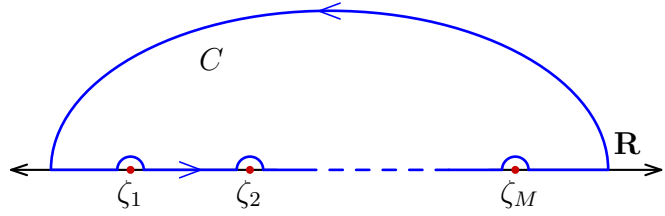}}
    \put(40,0){$\zeta_1$} \put(85,0){$\zeta_2$} \put(185,0){$\zeta_M$}
    \put(228,18){$\R$} \put(71,50){$C$}
  \end{picture}
 \caption{Contour in upper half plane}
 \label{F:contour}
\end{figure}
that contains semicircles of
radius $\epsilon$ centered at each root $\zeta_j$ and a semicircle of radius $1/\epsilon$
centered at 0, but otherwise lies along the real axis, for $\epsilon$ a sufficiently small
positive number.  

As $z$ moves along $C$, $\Arg\circ\Phi(z)$ takes on values $p_1,\dotsc,p_{M+1}$, for 
$z\in C\cap\R$.
On the semicircular arc around $\zeta_j$, it traces a curve from $p_{j}$ to $p_{j+1}$ in
which nearly every component is constant, except for those $i$ where $b_i(\zeta_j)=0$,
each of which decreases by $\pi$. 
In the limit as $\epsilon\to 0$, this becomes the line segment between $p_{j}$ and $p_{j+1}$
with direction $-\bbf_j$, where
\[
   \defcolor{\bbf_j}\ :=\ \sum_{i\colon b_i(\xi_j)=0} \be_i
   \ =\ \sum_{i=m_j}^{m_{j+1}-1} \be_i\,,
\]
and where we set $\be_{N+1}:=-(\be_1+\dotsb+\be_N)$.
This is because we are really working in the torus for $\P^N$, which is the quotient
$\T^{N+1}/\Delta(\T)$ of $\T^{N+1}$ modulo the diagonal torus, and
$\be_i\in\T^{N+1}/\Delta(\T)$ is the  image of the standard basis element in $\T^{N+1}$.
Thus $\be_1+\dotsb+\be_{N+1}=0$.

Along the arc near infinity, $\Arg\circ\Phi(z)$ approaches the line segment between $p_{M+1}$ and
$p_1$ which has direction $-\bbf_{M+1}$, where 
 \begin{equation}\label{Eq:bbf_M+1}
   \bbf_{M{+}1}\ =\ - \sum_{i\colon b_i(\infty)\neq0} \be_j\ =\ -(\bbf_1+\dotsb+\bbf_M)\,.
 \end{equation}
This polygonal path connecting $p_1,\dotsc,p_{M+1}$ in cyclic order forms the
boundary of the image of the upper half plane under $\Arg\circ\Phi$, which is a
two-dimensional membrane in $\T^N$.

The boundary of the image of the lower half plane is also a piecewise linear path
connecting $p_1,\dotsc,p_{M+1}$ in cyclic order, but the edge directions are
$\bbf_1,\dotsc,\bbf_{M{+}1}$.

\begin{example}\label{Ex:tc}
 Let $N=3$ and suppose that the affine functions $b_i$ are $z$, $1{-}2z$, $z{-}2$, 
 and $1$.
 Then $M=N$, $\xi_i=\zeta_i$, $\zeta_1=0$, $\zeta_1=1/2$, $\zeta_2=2$,
 and $\bbf_i=\be_i$. 
 The vertices of $\calA(\ell)$ are 
 \[
  p_1\ =\ (\pi,0,\pi)\,,\quad
  p_2\ =\ (0,0,\pi)\,,\quad 
  p_3\ =\ (0,-\pi,\pi)\,,\quad \mbox{and}\quad 
  p_4\ =\ (0,-\pi,0)\,.
 \]
 Figure~\ref{F:linecoamoeba} shows two views of $\calA(\ell)$ in the fundamental domain
 $[-\pi,\pi]^3\subset\R^3$ of $\T^3$, where the opposite faces of the cube are identified
 to form $\T^3$. 
 \hfill$\QED$
\end{example}
\begin{figure}[htb]
 \begin{picture}(133,115)
    \put(0,0){\includegraphics[height=115pt]{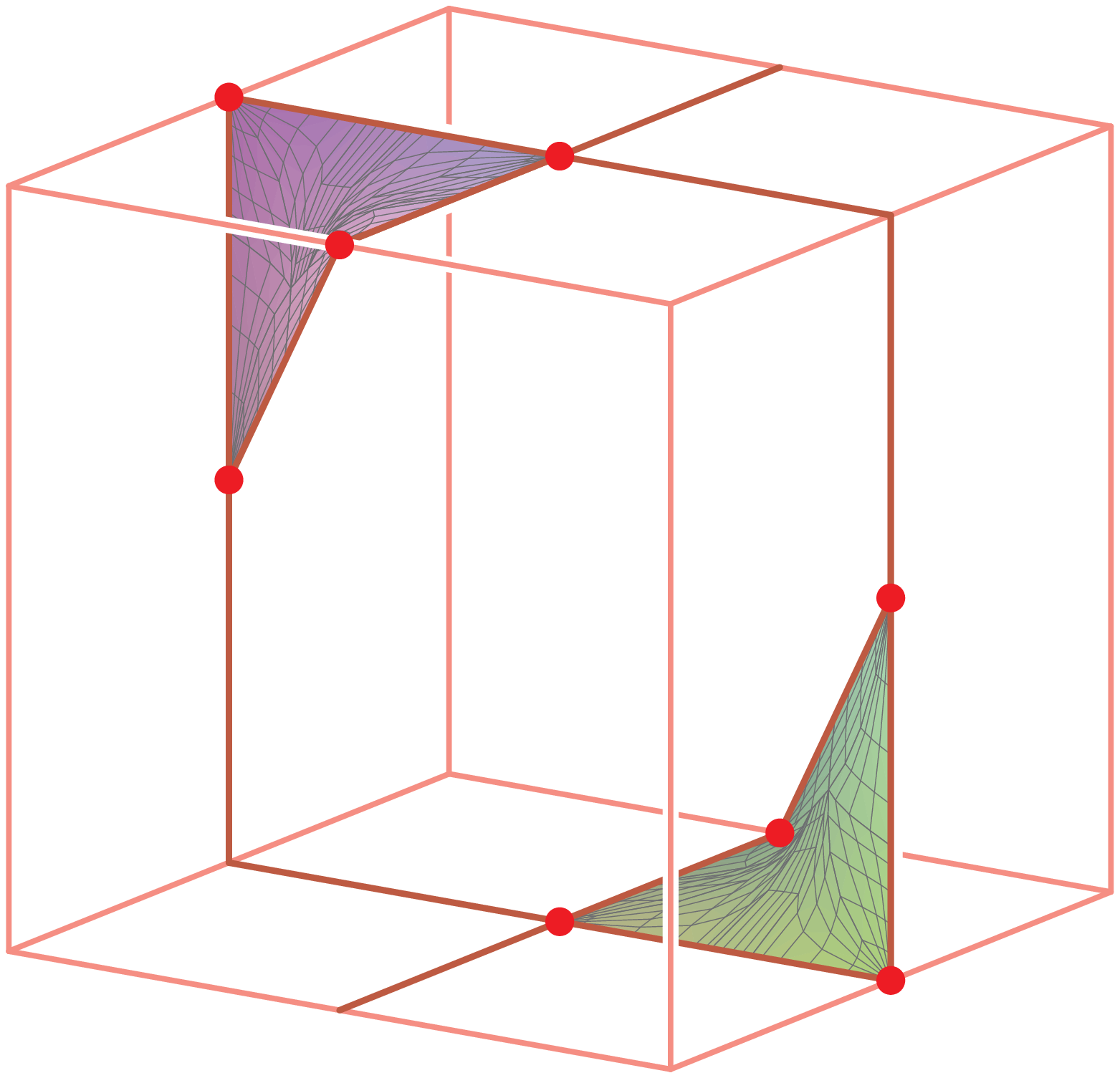}}
    \put(35, 78){$p_1$}  \put(53,104){$p_2$}
    \put(13,110){$p_3$}  \put(10,62){$p_4$}

    \put(75,33){$p_1$}   \put( 55, 7){$p_2$}
    \put(97, 2){$p_3$}   \put( 97,50){$p_4$}
  \end{picture}
  \qquad\qquad
  \begin{picture}(142,115)(-23,0)
    \put(0,0){\includegraphics[height=115pt]{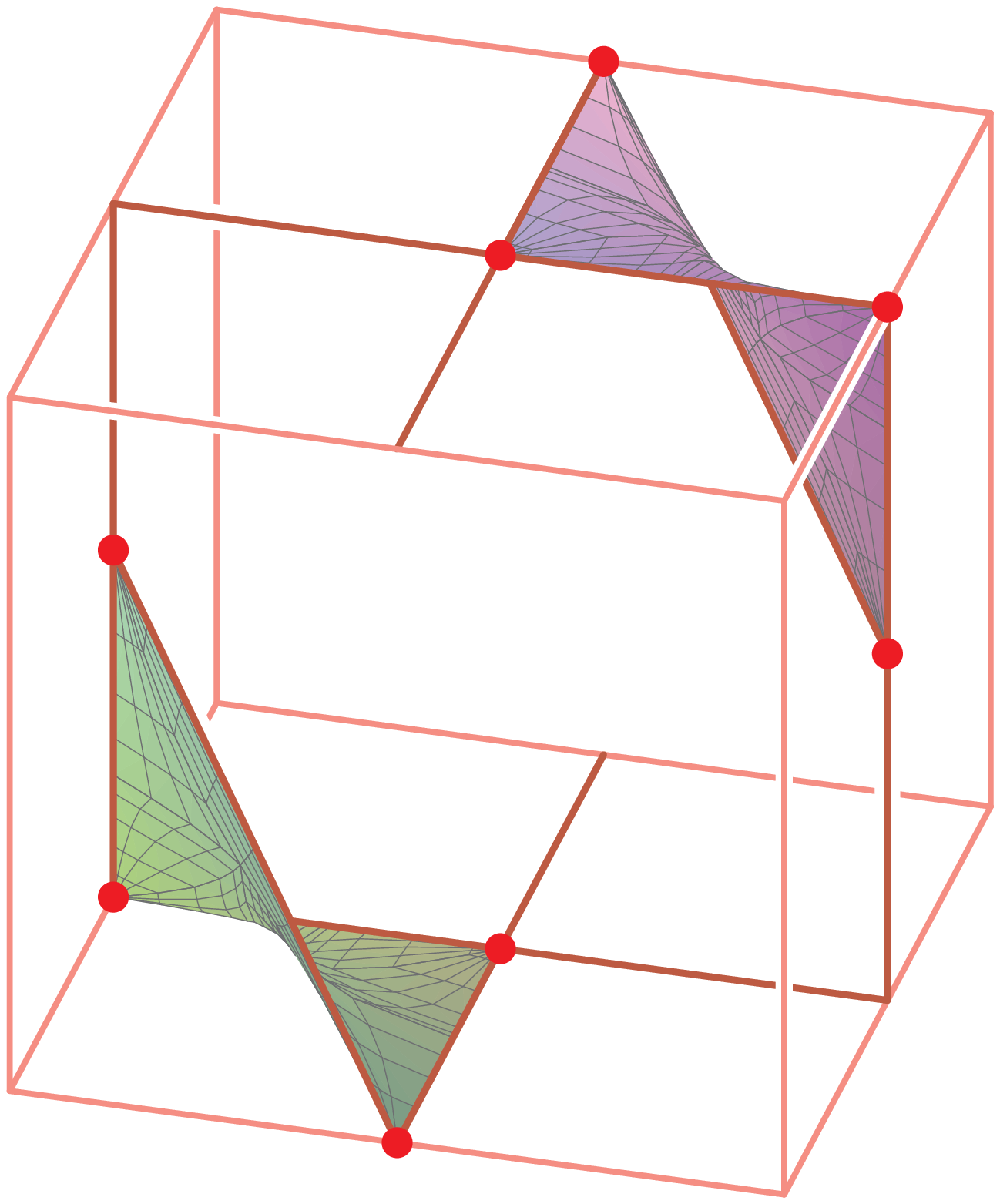}}
   \put( 43,104){$p_1$}  \put( 32, 84){$p_2$}
   \put(110, 83){$p_3$}  \put(108,86){\vector(-1,0){20}}
   \put(110, 50){$p_4$}  \put(108,53){\vector(-1,0){20}}

   \put( 30, -2){$p_1$}  \put( 36, 31){$p_2$}
   \put(-23, 28){$p_3$}  \put(-12,30){\vector(1,0){20}}
   \put(-23, 61){$p_4$}  \put(-12,63){\vector(1,0){20}}
  \end{picture}

 \caption{Two views of $\calA(\ell)$}
 \label{F:linecoamoeba} 
\end{figure}

\begin{example}\label{Ex:CoArepeat}
 We consider three examples when $N=3$ in which the affine functions have repeated zeroes.
 For the first, suppose that the affine functions $b_i$ are 
 $-1{-}z,-1{-}z,2z$, and $2$.
 These have zeroes $-1\leq-1<0<\infty$ and the vertices of the coamoeba $\calA(\ell)$ are 
\[
   (0,0,\pi)\,,\quad (-\pi,-\pi,\pi)\,,\quad \mbox{and}\quad
   (-\pi,-\pi,0)\,.
\]
 So $\calA(\ell)$ consists of two triangles with edges parallel to $\be_1{+}\be_2$,
 $\be_3$, and $\be_1{+}\be_2{+}\be_3$.
 It lies in the plane $\theta_1=\theta_2$.

 For a second example, suppose that the affine functions $b_i$ are 
 $\frac{1}{2}+z,\frac{1}{2}-z,-2$, and $1$.
 These have zeroes $-1,1,\infty$, and $\infty$.
 The vertices of the coamoeba $\calA(\ell)$ are 
\[
   (\pi,0,\pi)\,,\quad (0,0,\pi)\,,\quad \mbox{and}\quad
   (0,-\pi,\pi)\,.
\]
 So $\calA(\ell)$ consists of two triangles with edges parallel to $\be_1$,
 $\be_2$, and $\be_1{+}\be_2$.
 It lies in the plane $\theta_3=\pi$.

 Finally, suppose that the affine functions $b_i$ are 
 $-z,1-z,2z-2$, and $1$.
 These have zeroes $0,1,1$, and $\infty$.
 The vertices of the coamoeba $\calA(\ell)$ are 
\[
   (0,0,\pi)\,,\quad (-\pi,0,\pi)\,,\quad \mbox{and}\quad
   (-\pi,-\pi,0)\,.
\]
 So $\calA(\ell)$ consists of two triangles with edges parallel to $\be_1$,
 $\be_2{+}\be_3$, and $\be_1{+}\be_2{+}\be_3$.
 It lies in the plane $\theta_3=\theta_2+\pi$.
 We display all three coamoebas in Figure~\ref{F:repeat}.
 \hfill$\QED$
\end{example}

The \demph{coamoeba chain $\overline{\calA(\ell)}$} of $\ell$ is the closure of the
coamoeba of $\ell$ in which the image of each half plane (under  $\Arg\circ\Phi(\cdot)$) 
is oriented so that its boundary
is an oriented polygonal path connecting $p_1,\dotsc,p_{M+1},p_1$.
On the upper half plane this agrees with the orientation induced by the parametrization
$\P^1\smallsetminus\R\P^1\to\calA(\ell)$, but it has the opposite orientation on the lower
half plane. 
The boundary of $\overline{\calA(\ell)}$ consists of $M{+}1$ circles in which $p_j$ and
$p_{j+1}$ are antipodal points on the $j$th circle and both semicircles (each is the
boundary of the image of a half plane) are oriented to point from $p_j$ to $p_{j+1}$.
This coamoeba chain is not a closed chain, as it has nonempty oriented boundary, but there
is a natural zonotope chain $Z(\ell)$ such that $\overline{\calA(\ell)}+Z(\ell)$ is closed.

Intuitively, $Z(\ell)$ is the cone over the boundary of $\overline{\calA(\ell)}$ with vertex
the origin $\defcolor{\bO}:=(0,\dotsc,0)$. 
Unfortunately, there is no notion of a cone in $\T^N$ and the zonotope chain may be more than
just this cone.
We instead define a chain in $\R^N$ as the cone over an oriented
polygon $P(\ell)$ with vertex the origin and set $Z(\ell)$ to be the image of this chain in
$\T^N$. 

\begin{definition}\label{Def:Zonotope_chain}
 Recall that the affine functions $b_1,\dotsc,b_N,b_{N{+}1}=1$ are ordered in the
 following way.
 Their zeroes are $\zeta_1<\dotsb<\zeta_M<\zeta_{M{+}1}=\infty$ and there are 
 integers $1=m_1<\dotsb<m_{M+1}\leq N{+}1$ and $n_1,\dotsc,n_{M{+}1}$ with 
 $m_j< n_j\leq m_{j+1}$ such that one
 of~\eqref{Eq:inc} or~\eqref{Eq:dec} holds, where $\sgn_i$ is the sign of $b_i$ on
 $(-\infty,\zeta_1)$. 

 We had defined $\bbf_j:=\sum_{i=m_j}^{m_{j+1}-1}\be_i$.
 We will need the following vectors
\[
   \defcolor{\bg_j}\ :=\ \sum_{i=m_j}^{n_j-1}\be_i
     \qquad\mbox{and}\qquad
   \defcolor{\bh_j}\ :=\ \sum_{i=m_j}^{m_{j+1}-1}\sgn_i \be_i\ =\ 
     \sgn_{m_j}(2\bg_j-\bbf_j)\ \,.
\]

 We first define a sequence of points
 $\tp_1,\tpp_1,\dotsc,\tp_{2M+2},\tpp_{2M+2}\in(\pi\Z)^N$ with the property that
 $\tp_i,\tpp_i,\tp_{M{+}1{+}i},$ and $\tpp_{M{+}1{+}i}$ all map to $p_i\in\T^N$.
 To begin, set $\tp_1$ to be the unique point in $\{0,\pi\}^N\subset\R^N$ which maps to
 $p_1\in\T^N$,
 \begin{equation}\label{Eq:ptilde}
  \tp_{1,i}\ =\ \arg(\sgn_i)\ =\ \left\{
     \begin{array}{rcl}\pi&\ &\mbox{if }\sgn_i=-1\\
                         0&&\mbox{if }\sgn_i=1\end{array}\right.\ .
 \end{equation}

 For each $j=1,\dotsc,M{+}1$, set $\defcolor{\tp_{j+1}}:=\tp_j+\pi\bh_j$.
 Since $\bh_j=\sgn_{m_j}(2\bg_j-\bbf_j)$, we have that $\tp_{j+1}$ maps to $p_{j+1}$, as
 $p_{j+1}=p_j-\pi\bbf_j\mod (2\pi\Z)^N$.
 For the remainder of the points, if $n_j<m_{j+1}$, so that both signs occur, set 
 $\defcolor{\tpp_j} := \tp_j+2\pi\sgn_{m_j}\bg_j$, and otherwise set 
 $\defcolor{\tpp_j} := \tp_j$.
 Observe that $\tpp_j$ maps to $p_j$ and that
 in every case, $\tp_{j+1}=\tpp_j-\pi\sgn_{m_j}\bbf_j$.

 We claim that $\tp_{M{+}2}=-\tp_1$.
 Since $\tp_{M{+}2}=\tp_1+\pi(\bh_1+\dotsb+\bh_{M{+}1})$, we need to show that 
 $\pi(\bh_1+\dotsb+\bh_{M{+}1})=-2\tp_1$.
 By definition,
\[
  \bh_1+\dotsb+\bh_{M+1}\ =\ \sum_{i=1}^{N+1} \sgn_i\be_i\,.
\]
 We have $\sgn_{N+1}=1$ as $b_{N{+}1}=1$.
 Since we defined $\be_{N+1}$ to be $-(\be_1+\dotsb+\be_N)$, we see that
\[
  \bh_1+\dotsb+\bh_{M+1}\ =\ \sum_{i=1}^N (\sgn_i-1)\be_i\,.
\]
 The $i$th component of this sum is $-2$ if $\sgn_i=-1$ and $0$ if
 $\sgn_i=1$.
 Since $\tp_{1,i}=\arg(\sgn_i)$, this proves the claim.

 Finally, for each $M{+}2\leq j\leq 2M{+}2$, set
\[
    \defcolor{\tp_j}\ :=\ -\tp_{j-(M{+}1)}\qquad\mbox{and}\qquad
    \defcolor{\tpp_j}\ :=\ -\tp_{j-(M{+}1)}\,,
\]
 and let $P(\ell)$ be the cyclically oriented path obtained by connecting
\[
   \tpp_{2M+2}\,,\,\tp_{2M+2}\,,\,
   \tpp_{2M+1}\,,\,\tp_{2M+1}\,,\, \dotsc,
   \tpp_2\,,\,\tp_2\,,\,\tpp_1\,,\, \tp_1
\]
 in cyclic order.
 The cone over $P(\ell)$ with vertex the origin is the union of possibly degenerate triangles of
 the form 
\[
  \conv(\bO,\tp_{i+1},\tpp_{i})
  \qquad\mbox{and}\qquad  
  \conv(\bO,\tpp_{i},\tp_{i}) 
   \qquad\mbox{for}\qquad i=2M{+}2,\dotsc,2,1\,,
\]
where $\tp_{2M+3}:=\tp_1$. 
Each triangle is oriented so its three vertices  occur in positive order along its boundary.
If a point $\tp_i$ or $\tpp_i$ is $\bO$, then the triangles involving it degenerate into line
segments, as do triangles $\conv(\bO,\tpp_{i},\tp_{i})$ when $\tpp_{i}=\tp_{i}$.
Let $\widetilde{Z(\ell)}$ be the union of these oriented triangles, which is a chain in
$\R^N$. 
Define the \demph{zonotope chain $Z(\ell)$} to be the image in $\T^N$ of
$\widetilde{Z(\ell)}$. 
\hfill$\QED$
\end{definition} 

\begin{example}\label{Ex:CoZrepeat}
Figure~\ref{F:ZC} shows two views of the zonotope chain with the coamoeba chain of
Figure~\ref{F:linecoamoeba}. 
\begin{figure}[htb]
  \begin{picture}(153,150)
   \put(0,0){\includegraphics[height=150pt]{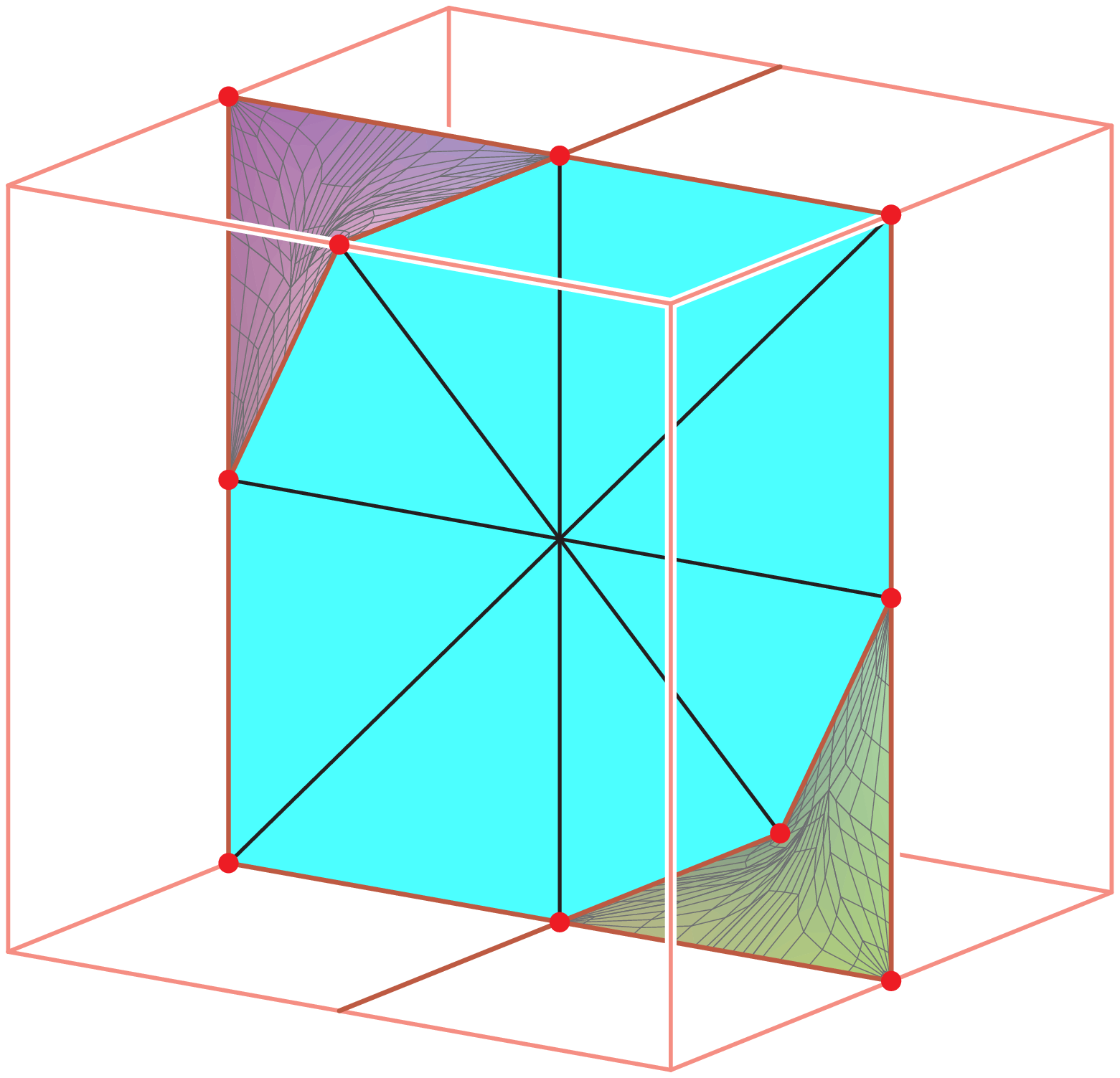}}
   \put( 44,101){$\tp_1$}  \put(69,134){$\tp_2$}
   \put(120,126){$\tp_3$}  \put(127,64){$\tp_4$}
   \put(101, 46){$\tp_5$}  \put(73,10){$\tp_6$}
   \put( 18, 32){$\tp_7$}  \put(18,82){$\tp_8$}
   \put(70.5,60){$\bO$}
  \end{picture}
   \qquad\qquad
  \begin{picture}(170,150)(-20,0)
   \put(0,0){\includegraphics[height=150pt]{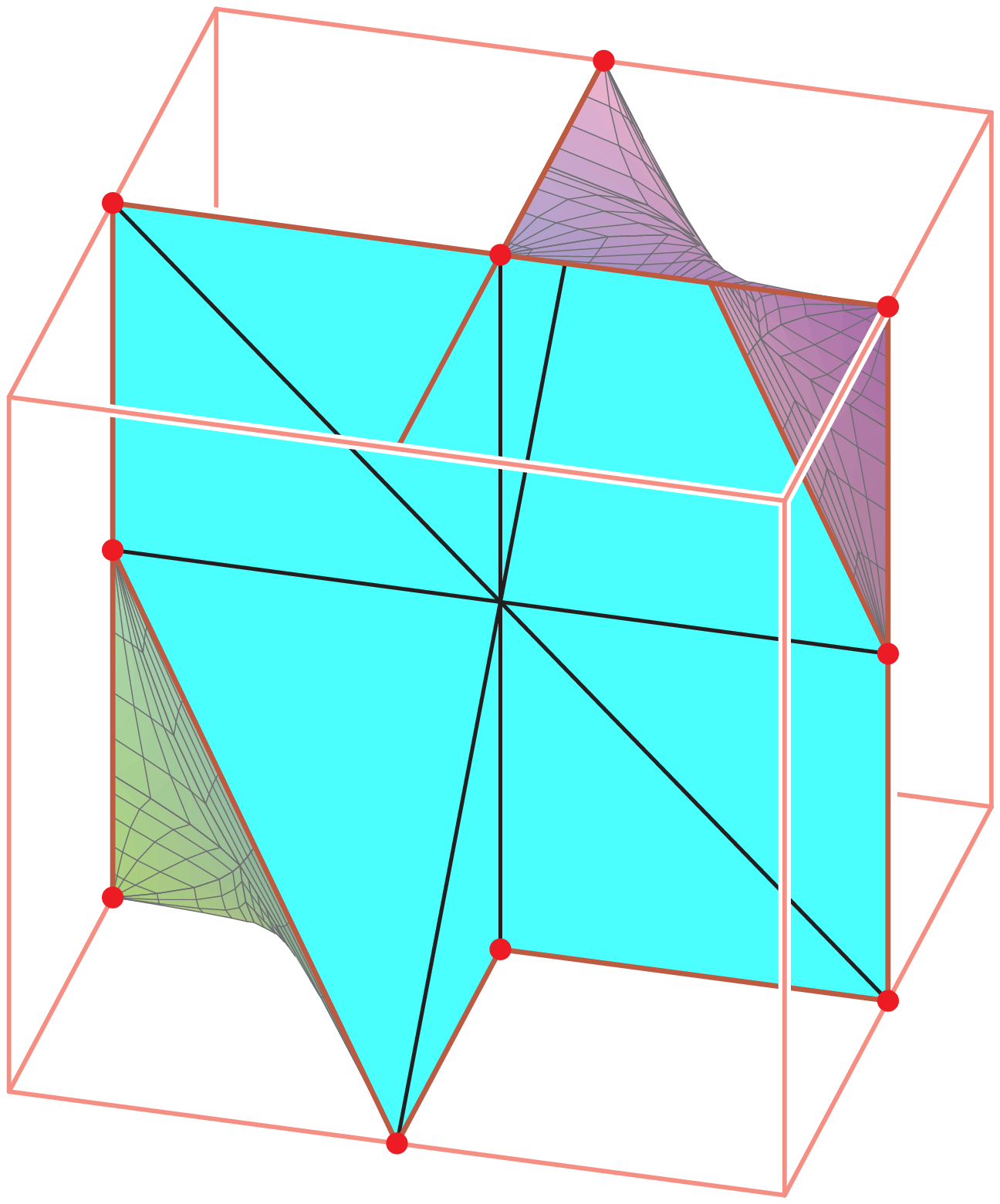}}
   \put( 59,135){$\tp_1$}  \put( 46,110){$\tp_2$}
   \put(  0,126){$\tp_3$} 
   \put(-20, 78){$\tp_4$} \put(-8,82){\vector(1,0){18}}
   \put( 55,  9){$\tp_5$}  \put( 65, 35){$\tp_6$}
   \put(113, 21){$\tp_7$} 
   \put(133, 66){$\tp_8$} \put(131,69){\vector(-1,0){18}}
   \put(53,66){$\bO$}
  \end{picture}
 \caption{Two views of the coamoeba and zonotope chains}
 \label{F:ZC}
\end{figure}
 Now consider the zonotope chains for the three lines of Example~\ref{Ex:CoArepeat}.
 When $\ell$ is defined by $z\mapsto[-1-z,-1-z,2z,2]$, the points $\tp_1,\dotsc,\tpp_6$
 (omitting repeated points) are
\[
  (0,0, \pi)\,,\ ( \pi, \pi, \pi)\,,\ ( \pi, \pi,0)\,,\ 
  (0,0,-\pi)\,,\ (-\pi,-\pi,-\pi)\,,\quad\mbox{ and }\quad (-\pi,-\pi,0)\,.
\]
We display the coamoeba chain and the zonotope chain of $\ell$ at the left of
Figure~\ref{F:repeat}. 
\begin{figure}[htb]

 \begin{picture}(102,150)(0,-20)
  \put(0,0){\includegraphics[height=110pt]{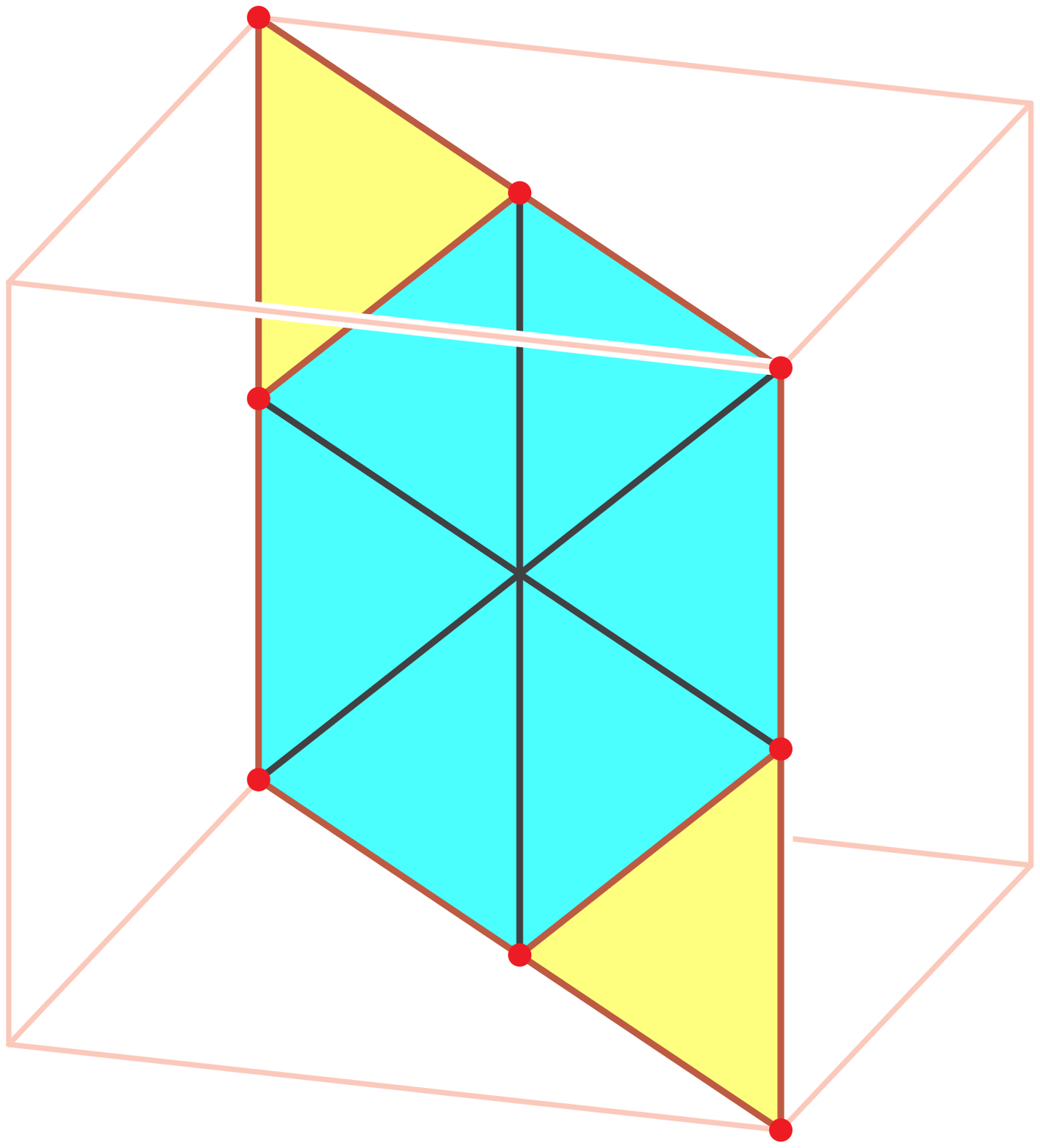}}
  \put( 51, 95){$\tp_1$}  \put(79, 72){$\tp_2$}
  \put( 79, 35){$\tp_3$}  
  \put( 39, 10){$\tp_4$}  \put( 11, 35){$\tp_5$}
  \put( 11, 70){$\tp_6$}  
  \thicklines
   \put( 12,118){\White{\vector(2,-3){21}}}
   \put( 90, -8){\White{\vector(-2,3){21}}}
  \thinlines
  \put(0 ,121){$\calA(\ell)$}\put( 12,118){\vector(2,-3){20}}
  \put(80,-18){$\calA(\ell)$}\put( 90, -8){\vector(-2,3){20}}
 \end{picture}
  \qquad
 \begin{picture}(116,150)(-5,-20)
  \put(0,0){\includegraphics[height=110pt]{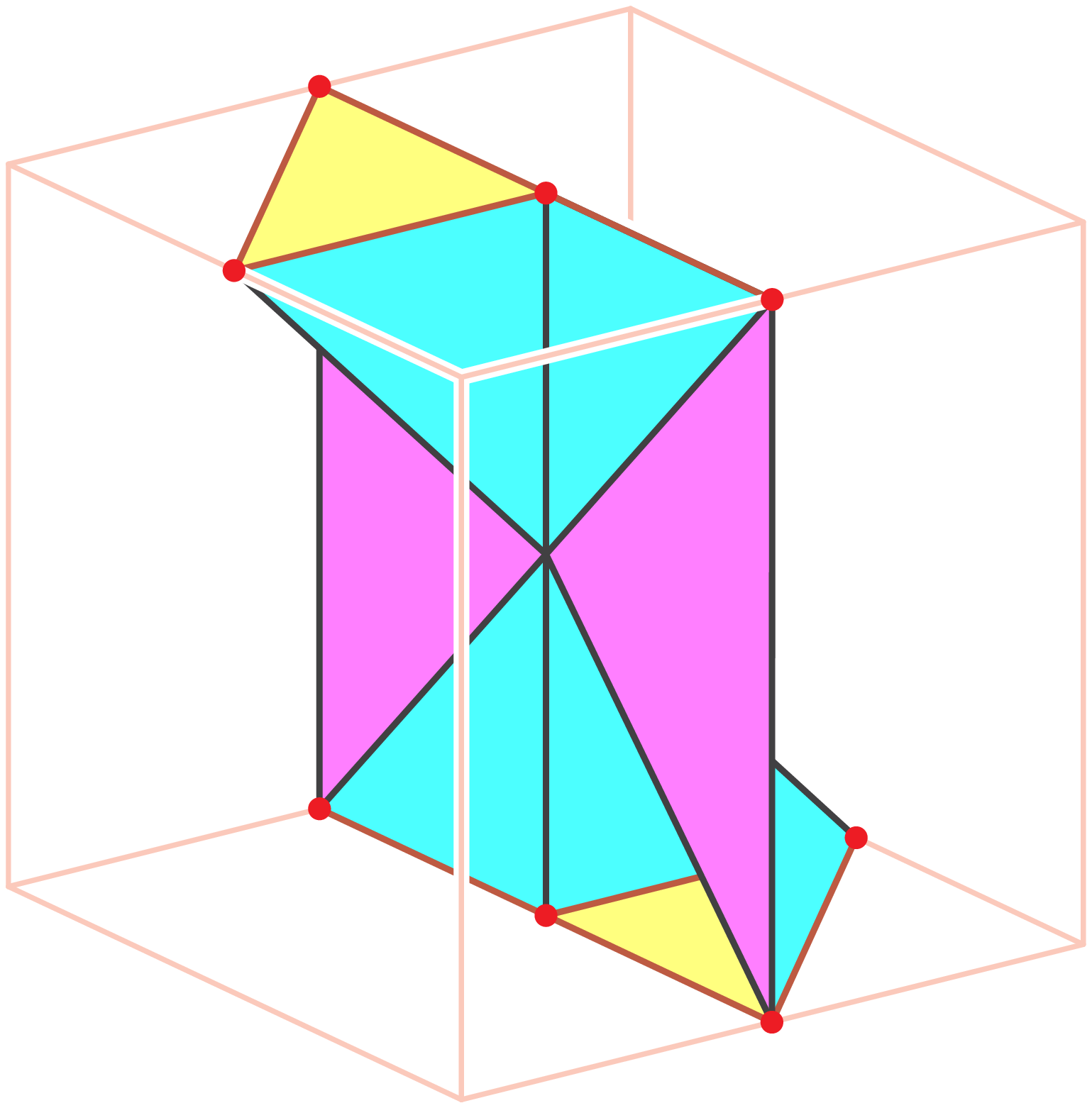}}
  \put(12, 75){$\tp_1$}   \put( 50,97){$\tp_2$}
  \put(79, 71){$\tp_3$}   \put( 72,-5){$\tpp_3$}
  \put(88, 30){$\tp_4$}   
  \put( 6,  4){$\tp_5$}  
  \put( 25,18){$\tp_6$}  \put( 25,107){$\tpp_6$}
  \put(-5,122){$\calA(\ell)$}
  \put( 28, -19){$\calA(\ell)$}
  \thicklines
   \put(16, 9){\White{\vector(4,1){36}}}
   \put(8,118){\White{\vector(1,-1){28}}}
   \put(42,-9){\White{\vector(1, 1){28}}}
  \thinlines
   \put(16,9){\vector(4,1){35}}
   \put(8,118){\vector(1,-1){27}}
   \put(42,-9){\vector(1, 1){27}}
 \end{picture}
  \qquad
 \begin{picture}(181,150)(-5,-20)
  \put(0,0){\includegraphics[height=110pt]{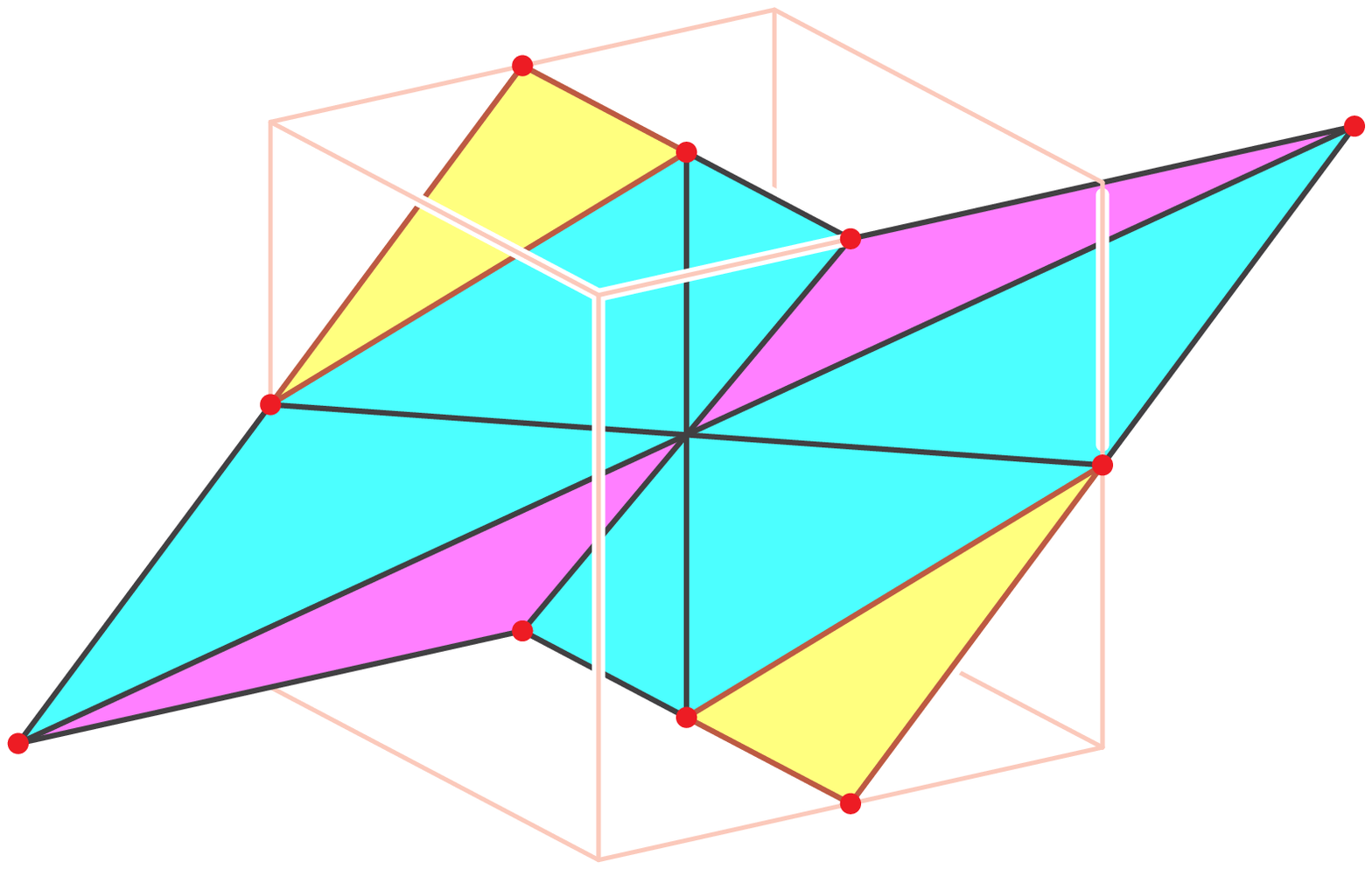}}
  \put( 83,  97){$\tp_1$}   \put(104, 89){$\tp_2$}
  \put(166, 101){$\tpp_2$}  \put(144, 47){$\tp_3$}
  \put( 81,   7){$\tp_4$}   
  \put( 60, 19){$\tp_5$} 
  \put( -5,  3){$\tpp_5$}   \put( 20, 64){$\tp_6$}
  \put(  0,120){$\calA(\ell)$}
  \put(145, -10){$\calA(\ell)$}
  \thicklines
   \put( 18,114){\White{\vector(2,-1){49}}}
   \put(153,  0){\White{\vector(-2,1){49}}}
  \thinlines
   \put( 18,114){\vector(2,-1){48}}
   \put(153,  0){\vector(-2,1){48}}
 \end{picture}

\caption{Coamoeba and zonotope chains}
\label{F:repeat}
\end{figure}

When $\ell$ is defined by $z\mapsto [\frac{1}{2}+z,\frac{1}{2}-z,-2,1]$, the points
$\tp_1,\dotsc,\tpp_6$ are  
 \begin{eqnarray*}
  &\tp_1\ =\ (\pi,0,\pi)\,,\ 
  \tp_2\ =\ (0,0,\pi)\,,\ 
  \tp_3\ =\ (0,\pi,\pi)\,,\ 
  \Magenta{\tpp_3\ =\ (0,\pi,-\pi)}\,,&\\ 
  &\tp_4\ =\ (-\pi,0,-\pi)\,,\ 
  \tp_5\ =\ (0,0,-\pi)\,,\ 
  \tp_6\ =\ (0,-\pi,-\pi)\,,\ 
  \Magenta{\tp_6'\ =\ (0,-\pi,\pi)}\,.&
 \end{eqnarray*}
We display the coamoeba and zonotope chains of $\ell$ in the middle of Figure~\ref{F:repeat}.

When $\ell$ is defined by $z\mapsto [-z,1-z,2z-2,1]$, the
points $\tp_1,\dotsc,\tpp_6$  are 
 \begin{eqnarray*}
  &\tp_1\ =\ (0,0,\pi)\,,\ 
  \tp_2\ =\ (\pi,0,\pi)\,,\ 
  \Magenta{\tpp_2\ =\ (\pi,2\pi,\pi)}\,,\ 
  \tp_3\ =\ (\pi,\pi,0)\,,&\\ 
  &\tp_4\ =\ (0,0,-\pi)\,,\ 
  \tp_5\ =\ (-\pi,0,-\pi)\,,\ 
  \Magenta{\tpp_5\ =\ (-\pi,-2\pi,-\pi)}\,,\ 
  \tp_6\ =\ (-\pi,-\pi,0)\,.&
 \end{eqnarray*}
We display the coamoeba and zonotope chains of $\ell$ on the right of Figure~\ref{F:repeat}.
\hfill$\QED$
\end{example}

We state the main result of this section.

\begin{theorem}\label{Th:homology}
  The sum, $\overline{\calA(\ell)}+Z(\ell)$, of the coamoeba chain and the zonotope chain forms a
  cycle in $\T^N$ whose homology class is
\[
   [\overline{\calA(\ell)}+Z(\ell)]\ =\ 
    \sum_{\substack{1\leq i<j\leq N\\(\tp_{1,i},\tp_{1,j})=(0,\pi)}}
    \be_i\wedge\be_j\,.
\]
\end{theorem}

\begin{example}\label{Ex:homology_class}
For the line of Example~\ref{Ex:tc}, $\tp_1=(\pi,0,\pi)$, and the only entries $i<j$ with 0 at
$i$ and $\pi$ at $j$ are $i=2$ and $j=3$, and so
\[
    [\overline{\calA(\ell)}+Z(\ell)]\ =\ \be_2\wedge\be_3\,.
\]
For the first line of Example~\ref{Ex:CoArepeat}, $\tp_1=(0,0,\pi)$, and so 
\[
    [\overline{\calA(\ell)}+Z(\ell)]\ =\ \be_1\wedge\be_3 + \be_2\wedge\be_3\,.
\]
For the second line of Example~\ref{Ex:CoArepeat}, $\tp_1=(\pi,0,\pi)$, so 
that $[\overline{\calA(\ell)}+Z(\ell)]=\be_2\wedge\be_3$.
For the third line of Example~\ref{Ex:CoArepeat}, $\tp_1=(0,0,\pi)$, and 
$[\overline{\calA(\ell)}+Z(\ell)]=\be_1\wedge\be_3 + \be_2\wedge\be_3$.
These homology classes are apparent from Figures~\ref{F:ZC} and~\ref{F:repeat}.
 \hfill$\QED$
\end{example}
  
\begin{example}\label{Ex:plane}
 Our proof of Theorem~\ref{Th:homology} rests on the case of $N=2$.
 Suppose first that $M=2$.
 Up to positive rescaling and translation in the domain $\R\P^1$, there are four lines.
\[
  \begin{picture}(65,62)(-7.5,-12)
     \put(0,0){\includegraphics{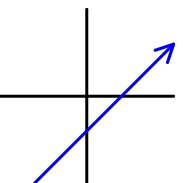}}
     \put(-4.5,-12){$[z:z{-}1:1]$}
   \end{picture}
  \qquad
   \begin{picture}(65,62)(-7.5,-12)
     \put(0,0){\includegraphics{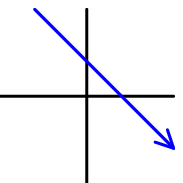}}
     \put(-4.5,-12){$[z:1{-}z:1]$}
   \end{picture}
  \qquad
    \begin{picture}(65,62)(-7.5,-12)
     \put(0,0){\includegraphics{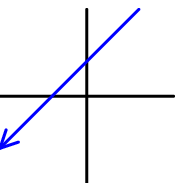}}
     \put(-8.5,-12){$[-z:1{-}z:1]$}
   \end{picture}
  \qquad
    \begin{picture}(65,62)(-7.5,-12)
     \put(0,0){\includegraphics{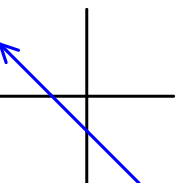}}
     \put(-8.5,-12){$[-z:z{-}1:1]$}
   \end{picture}
\]
 For these, the initial point 
 $p_1$ is $(\pi,\pi)$, $(\pi,0)$, $(0,0)$, and $(0,\pi)$, respectively.
 The four coamoeba chains are, in the fundamental domain $[-\pi,\pi]^2$, 
\[
   \includegraphics{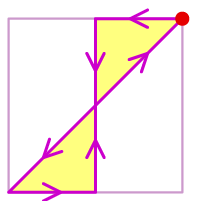}
  \qquad
   \includegraphics{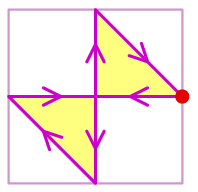}
  \qquad
   \includegraphics{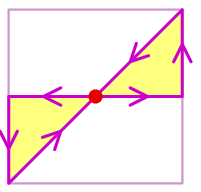}
  \qquad
   \includegraphics{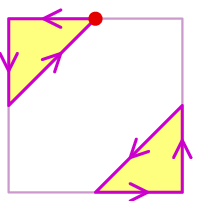}
\]
 and the corresponding zonotope chains are as follows.
\[
   \includegraphics{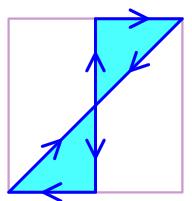}
  \qquad
   \includegraphics{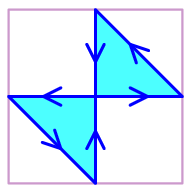}
  \qquad
   \includegraphics{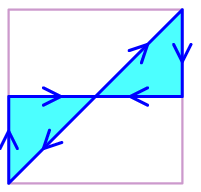}
  \qquad
   \includegraphics{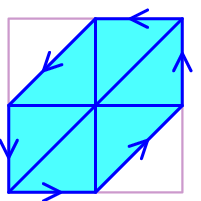}
\]
 For each, the sum $\overline{\calA(\ell)}+Z(\ell)$ of chains is a
 cycle.
 This cycle is homologous to zero for the first three, and it forms the
 fundamental cycle $\be_1\wedge\be_2$ of $\T^2$ for the fourth.

 Now suppose that $M=1$.
 We may assume that $\xi_1=0$.
 Up to positive rescaling there are eight
 possibilities for the parametrization of $\ell$,
 \begin{eqnarray*}
   &[-z:-z:1]\,,\ 
   [ z: z:1]\,,\ 
   [-z: 1:1]\,,\ 
   [ z: 1:1]\,,&\\
   &[ z:-z:1]\,,\ 
   [ z:-1:1]\,,\ 
   [-z:-1:1]\,,\ 
   [-z: z:1]\,.&
 \end{eqnarray*}
 For all of these, the coamoeba is one-dimensional.
 In the first four, the zonotope chain is one-dimensional.
 Table~\ref{T:paths} gives the parametrization, the vertices of the coamoeba of the upper half plane, and
 the path $P(\ell)=\tp_4,\tp_3,\tp_2,\tp_1$ for these four.
\begin{table}[htb]
 \caption{Coamoeba and zonotope chains.}\label{T:paths}
 \begin{tabular}{|r||l|l|}\hline
  \multicolumn{1}{|c||}{$\ell$} &\multicolumn{1}{|c|}{$\calA(\ell)$}&\multicolumn{1}{|c|}{$P(\ell)$}\\\hline
  $[-z:-z:1]$&$(0,0)\,,\,(-\pi,-\pi)$&$(-\pi,-\pi)\,,\,(0,0)\,,\,(\pi,\pi)\,,\,(0,0)$\\\hline
  $[ z: z:1]$&$(\pi,\pi)\,,\,(0,0)  $&$(0,0)\,,\,(-\pi,-\pi)\,,\,(0,0)\,,\,(\pi,\pi)$\\\hline
  $[-z: 1:1]$&$(0,0)\,,\,(-\pi,0)   $&$(-\pi,0)\,,\,(0,0)\,,\,(\pi,0)\,,\,(0,0)$\\\hline
  $[ z: 1:1]$&$(\pi,0)\,,\,(0,0)    $&$(0,0)\,,\,(-\pi,0)\,,\,(0,0)\,,\,(\pi,0)$\\\hline
 \end{tabular}
\end{table}

 The remaining parametrizations are more interesting.
 When $\ell$ is given by $z\mapsto[z:-z:1]$, we have $p_1=(\pi,0)$ and $p_2=(0,-\pi)$, and
 $P(\ell)$ is 
\[
   \tp_4=(0,-\pi)\,,\ \tpp_3=(\pi,0)\,,\ \tp_3=(-\pi,0)\,,\ 
   \tp_2=(0,\pi)\,,\ \tpp_1=(-\pi,0)\,,\quad\mbox{and}\quad\tp_1=(\pi,0)\,,
\]
 and the zonotope chain is shown on the left in Figure~\ref{F:M=1}.
 The path $\tp_4{-}\tpp_3{-}\tp_3{-}\tp_2{-}\tpp_1{-}\tp_1{-}\tp_4$ zig-zags over itself, once
 in each direction, and consequently each triangle is covered twice, once with each
 orientation, and therefore $[Z(\ell)]=0$ in homology.
\begin{figure}[htb]

\begin{picture}(90,89)(-23,-24)
  \put(0,0){\includegraphics{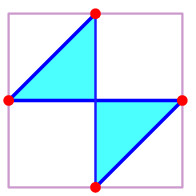}}
  \put( 42,33){$\tp_1=\tpp_3$}   
  \put(-23,15){$\tpp_1=\tp_3$}
  \put(23,-9){$\tp_4$}   \put(23,58){$\tp_2$}
  \put(0,-24){$[ z:-z:1]$}
\end{picture}
\qquad
\begin{picture}(81,89)(-15,-24)
  \put(0,0){\includegraphics{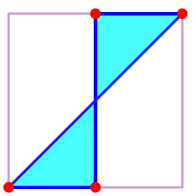}}
  \put(-11,-7){$\tp_3$} \put(57,52){$\tp_1$}
  \put(21,-9){$\tpp_2{\,=\,}\tp_4$}
  \put(-2,58){$\tp_2{\,=\,}\tpp_4$}
  \put(0,-24){$[ z:-1:1]$}
\end{picture}
\qquad
\begin{picture}(81,89)(-15,-24)
  \put(0,0){\includegraphics{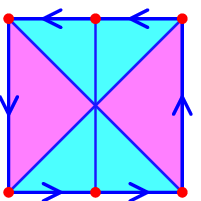}}
  \put(-11,55){$\tpp_4$} \put(-11,-7){$\tp_4$}
  \put(22,-9){$\tp_3$}  \put(57,-7){$\tpp_2$}
  \put(57,52){$\tp_2$}  \put(22,58){$\tp_1$}
  \put(0,-24){$[-z:-1:1]$}
\end{picture}
\qquad
\begin{picture}(120,89)(-32,-24)
  \put(0,0){\includegraphics{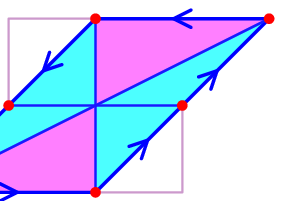}}
  \put(-11,30){$\tp_4$} \put(-36,-3){$\tpp_3$}
  \put(23,-9){$\tp_3$}  \put(56,19){$\tp_2$}
  \put(81,52){$\tpp_1$}  \put(22,58){$\tp_1$}
  \put(0,-24){$[-z: z:1]$}
\end{picture}

\caption{Four more zonotope chains.}\label{F:M=1}

\end{figure}

 When $\ell$ is given by $[z:-1:1]$, we have $p_1=(\pi,\pi)$ and $p_2=(0,\pi)$, and
 $P(\ell)$ is 
\[
   \tpp_4=(0,\pi)\,,\ \tp_4=(0,-\pi)\,,\ \tp_3=(-\pi,-\pi)\,,
   \ \tpp_2=(0,-\pi)\,,\ \tp_2=(0,\pi)\,,\quad\mbox{and}\quad\tp_1=(\pi,\pi)\,,
\]
 and the zonotope chain is shown on the left center of Figure~\ref{F:M=1}.
 As before, each triangle is covered twice, once with each orientation, and therefore
 $[Z(\ell)]=0$ in homology.

 When $\ell$ is given by $[-z:-1:1]$, we have $p_1=(0,\pi)$ and
 $p_2=(-\pi,\pi)$, and $P(\ell)$ is 
\[
   \tpp_4=(-\pi,\pi)\,,\ \tp_4=(-\pi,-\pi)\,,\  \tp_3=(0,-\pi)\,,
   \tpp_2=(\pi,-\pi)\,,\ \tp_2=(\pi,\pi)\,,\quad\mbox{and}\quad  \tp_1=(0,\pi)\,,
\]
 and the zonotope chain is shown on the right center of Figure~\ref{F:M=1}.
 The triangles $\conv(\bO,\tp_2,\tpp_2)$ and $\conv(\bO,\tp_4,\tpp_4)$ are shaded
 differently.
 The zonotope chain is equal to the fundamental cycle of $\T^2$, with the standard
 positive orientation.
 Thus $[Z(\ell)]=\be_1\wedge\be_2$ in homology.

 Finally, when $\ell$ is given by $[-z: z:1]$, we have $p_1=(0,\pi)$ and
 $p_2=(-\pi,0)$, and $P(\ell)$ is 
\[
   \tp_4=(-\pi,0)\,,\ \tpp_3=(-2\pi,-\pi)\,,\  \tp_3=(0,-\pi)\,,
   \tp_2=( \pi,0)\,,\ \tpp_1=( 2\pi, \pi)\,,\quad\mbox{and}\quad  \tp_1=(0,\pi)\,,
\]
 and the zonotope chain is shown on the right of Figure~\ref{F:M=1}.
 Again, $[Z(\ell)]=[\T^2]$.

 Observe that $\calA(\ell)+Z(\ell)$ forms a cycle which is homologous to zero
 unless $\tp_1=(0,\pi)$, in which case it equals the fundamental cycle $\be_1\wedge\be_2$ of $\T^2$.
 \hfill$\QED$
\end{example}

\begin{proof}[Proof of Theorem~$\ref{Th:homology}$]
 We show that the two chains $\overline{\calA(\ell)}$ and $Z(\ell)$ have
 the same boundary, but with opposite orientation, which implies that their sum is a cycle.
 We observed that the boundary of $\calA(\ell)$ lies along the $M{+}1$ circles in which
 the $j$th contains $p_j$ and $p_{j+1}$ (with $p_{M+2}=p_1$) and has direction parallel to
 $\bbf_j$.
 On this $j$th circle the boundary of $\calA(\ell)$ consists of the two semicircles
 oriented from $p_j$ to $p_{j+1}$.

 There are two types of edges forming the boundary of the zonotope cycle $Z(\ell)$.
 The first comes from the edges of $P(\ell)$ with direction $\pm\bbf_j$ connecting
 $\tp_{j+1}$ to $\tpp_j$ and  $\tp_{M+1+j+1}$ to $\tpp_{M+1+j}$, and the second
 comes from edges connecting $\tpp_j$ to $\tp_j$, when $\tpp_j\neq\tp_j$.

 The first type of edge gives a part of the boundary of $Z(\ell)$ which is equal to the
 boundary of $\calA(\ell)$, but with opposite orientation.
 (The edges point from $p_{j+1}$ to $p_j$.)
 The edges of the second type come in pairs which cancel each other.
 Indeed, when $\tp_j\neq\tpp_j$, then the edge from $\tpp_j$ to $\tp_j$ is the directed
 circle connecting $p_j$ with itself and having direction $\pm\bg_j$, which is equal to, but
 opposite from, the edge connecting $\tpp_{M+1+j}$ to $\tp_{M+1+j}$.
 Thus $\overline{\calA(\ell)}+Z(\ell)$ forms a cycle in homology.

 We determine the homology class $[\overline{\calA(\ell)}+Z(\ell)]$ by computing its pushforward
 to each two-dimensional coordinate projection of $\T^N$.
 Let $1\leq i<j\leq N$ be two coordinate directions and consider 
 the projection onto the plane of the coordinates $i$
 and $j$, which is a map $\pr\colon\T^N\to\T^2$.
 The image of $\ell$ under $\pr$ is parametrized by 
 \begin{equation}\label{Eq:proj_param}
  z\ \longmapsto\ 
  [b_i(z)\,:\,b_j(z)\,:\,b_{N+1}(z)]\,.
 \end{equation}

 If $b_i,b_j$, (and $b_{N{+}1}=1$) all vanish at $\xi_{N{+}1}=\infty$, then the image of
 $\ell$ under $\pr$ is a point, and the image of $Z(\ell)$ is either a point or is
 one-dimensional, and so $\pr_*[\calA(\ell)+Z(\ell)]=0$.
 In this case $(\tp_{1,i},\tp_{1,j})$ is either $(0,0)$, $(\pi,0)$, or $(\pi,\pi)$,
 by~\eqref{Eq:inc} and~\eqref{Eq:ptilde}.

 Otherwise, the image of $\ell$ under the
 projection of $\P^N$ to the $(i,j)$-coordinate plane is the line
 $\ell'$ parameterized by~\eqref{Eq:proj_param}.
 It is immediate from the definitions that 
\[
   \pr(\overline{\calA(\ell)})\ = \overline{\calA(\ell')}
    \qquad\mbox{and}\qquad
   \pr(Z(\ell))\ =\ Z(\ell')\,.
\]
 When $b_i$ and $b_j$ have distinct (finite) zeroes, say $\zeta_a$ and $\zeta_b$, then $\pr$ is
 injective on the interior of $\overline{\calA(\ell)}$ and on the edges with directions 
 $\pm\bbf_a$, $\pm\bbf_b$, and $\pm\bbf_{M+1}$ (sending them to edges with directions
 $\pm\be_1$, $\pm\be_2$, and $\pm(\be_1{+}\be_2)$) and collapsing the others
 to points. 
 In the other cases, $\calA(\ell')$ is a circle.
 However, in all cases $\pr$ is one-to-one over the interiors of each triangle in the
 image zonotope cycle  $Z(\ell')$, collapsing the other triangles to line segments or to points.
 Thus 
\[
   \pr_*[\overline{\calA(\ell)}+Z(\ell)]\ =\ 
   [\overline{\calA(\ell')}+Z(\ell')]\,.
\]
 Since the last vertex of the path $P(\ell')$ is $(\tp_{1,i},\tp_{1,j})$, the theorem
 follows from the computation of Example~\ref{Ex:plane}.
\end{proof}

%
\section{Structure of discriminant coamoebas in dimension two}\label{S:three}

Suppose now that $B\subset\Z^2$ is a multiset of $N{+}1$ vectors which span $\R^2$ and
have sum $\bO=(0,0)$.
We use $B=\{\bb_1,\dotsc,\bb_{N{+}1}\}$ to define a rational map $\C^2-\to\C^2$ 
 \begin{equation}\label{Eq:HK}
  z\ \longmapsto\ 
  \Big( \prod_{i=1}^{N+1} \langle \bb_i,z\rangle^{\bb_{i,1}}\,,\, 
        \prod_{i=1}^{N+1} \langle \bb_i,z\rangle^{\bb_{i,2}}\Bigr)\,.
 \end{equation}
Since $\sum_i \bb_i=\bO$, each coordinate is homogenous of degree $0$, and
so~\eqref{Eq:HK} induces a rational map $\Psi_B\colon\P^1\to\P^2$ (where the image has
distinguished coordinates).  
Define \defcolor{$D_B$} to be the image of this map~\eqref{Eq:HK}.
When $B$ consists of distinct vectors that span $\Z^2$, then it is Gale dual to a
set of vectors of the form $(1,\ba)$ for $\ba\in A\subset\Z^{n+2}$.
In this case,~\eqref{Eq:HK} is the Horn-Kapranov parametrization~\cite{K91} of the reduced
$A$-discriminant. 
We use Theorem~\ref{Th:homology} to study the coamoeba \defcolor{$\calA_B$} of $D_B$ and
its complement, for any multiset $B$. 

The results of Section~\ref{S:realline} are applicable because the map~\eqref{Eq:HK} factors,
\[
  \begin{array}{rcrcl}
   \C^2\ni z&\longmapsto&(\langle\bb_1,z\rangle,\langle\bb_2,z\rangle,
               \dotsc,\langle\bb_{N+1},z\rangle)\in\C^{N+1}\\
   &&\C^{N+1}\ni (x_1,x_2,\dotsc,x_{N+1})&\longmapsto&
     {\displaystyle\Big( \prod_{i=1}^{N+1} x_i^{\bb_{i,1}}\,,\, 
           \prod_{i=1}^{N+1} x_i^{\bb_{i,2}}\Bigr)\in\C^2}
  \end{array}
\]
The first map, $\Phi_B$, is linear and the second, $\beta$, is a monomial map.
They induce maps $\P^1\to\P^N-\to\P^2$, with the second a rational map.
Let \defcolor{$\ell_B$} be the image of $\Phi_B$ in $\P^N$, which is a real line as in 
Section~\ref{S:realline}.
The map $\Arg(\beta)$ is the homomorphism $\T^N\to\T^2$ induced by the linear map 
on the universal covers, (also written $\Arg(\beta)$),
 \[
  \Arg(\beta)\ \colon\ \R^N\ \ni\ \be_i\ \longmapsto\ \bb_i\ \in\ \R^2\,,
 \]
and the following is immediate.

\begin{lemma}\label{L:coamoeba_structure}
   The coamoeba $\calA_B$ is the image of the coamoeba $\calA(\ell_B)$ under 
   the map $\Arg(\beta)$.
\end{lemma}

\begin{example}\label{Ex:rational_cubic}
 Let $B$ be the vector configuration $\{(1,0), (-2,1), (1,-2), (0,1)\}$.
 Observe that $\bb_1+\bb_2+\bb_3+\bb_4=0$ and $3\bb_1+2\bb_2+\bb_3=0$, thus $B$ is Gale dual
 to the vector configuration $\{(1,3),(1,2),(1,1),(1,0)\}\subset\{1\}\times\Z$.
 So $A$ is simply $\{0,1,2,3\}$ if we identify $\Z$ with $\{1\}\times\Z$.
We show these two configurations.
\[
  \begin{picture}(63,58)(-10,-4)
   \put(0,0){\includegraphics{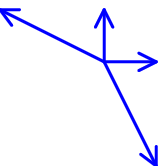}}
   \put( 27,47){$\bb_4$}   \put(47,26){$\bb_1$}
   \put(-3,50){$\bb_2$}   \put(47,-4){$\bb_3$}
   \put(5,5){$B$}
  \end{picture}
   \qquad\qquad
  \begin{picture}(68,59)(-2,0)
   \put(0,30){\includegraphics{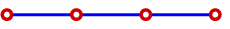}}
   \put(-1.5,23.4){$0$} \put(18.8,23.4){$1$}
   \put(38.8,23.4){$2$} \put(58.8,23.4){$3$}
   \put(29,44){$A$}
  \end{picture}
\]
 Observe that the convex hull of $A$ has volume $d_B=3$.
 
 The map~\eqref{Eq:HK} becomes
\[
   (x,y)\ \longmapsto\ \Bigl( \frac{x(x-2y)}{(y-2x)^2}\,,\,
          \frac{y(y-2x)}{(x-2y)^2}\Bigr)\,,
\]
 whose image is the curve below.
\[
  \begin{picture}(81,100)
   \put(0,0){\includegraphics[height=100pt]{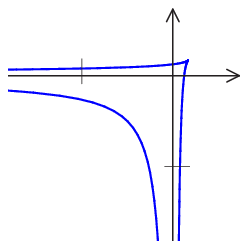}}
   \put(21,76.5){$-2$}   \put(78,29){$-2$}
  \end{picture}
\]
 The line $\ell_B$ is the line of Example~\ref{Ex:tc} and so $\calA_B$ is the image of the
 coamoeba of Figure~\ref{F:linecoamoeba} under the map 
\[
   \Arg(\beta)\ \colon\ (\theta_1,\theta_2,\theta_3)\ \longmapsto\ 
    (\theta_1{-}2\theta_2{+}\theta_3, \theta_2{-}2\theta_3)\,.
\]
 We display this image below, first in the fundamental domain $[\pi,\pi]^2$ of $\T^2$, and
 then in  universal cover $\R^2$ of $\T^2$ (each square is one fundamental domain). 
 \begin{equation}\label{Eq:A-discr}
  \raisebox{-40pt}{\begin{picture}(60,90)
   \put(0,0){\includegraphics{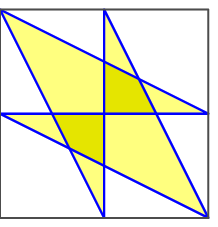}}
   \put(23,2){$\calA_B$}
  \end{picture}
   \qquad\qquad
  \begin{picture}(165,90)(5,0)
   \put(37,0){\includegraphics{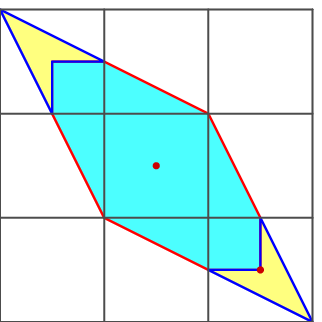}}
   \put(0,2){$\calA_B$}
   \put(20,5){\vector(1,0){92}}
   \put(11,13){\vector(1,2){31}}
   \put(84,46){$0$}
   \put(105,75){$Z_B$} \put(104,78){\vector(-2,-1){30}}
  \end{picture}}
 \end{equation}
 In the picture on the left, the darker shaded regions are where the argument map is
 two-to-one.
 The octagon on the right is the zonotope $Z_B$ generated by $B$ and it is the image of the
 zonotope chain of Figure~\ref{F:ZC} under the map $\Arg(\beta)$. 
 Observe that the union of the coamoeba and the zonotope covers the fundamental domain $d_B=3$
 times.
 \hfill$\QED$
\end{example}

What we observe in this example is in fact quite general.
We first use Lemma~\ref{L:coamoeba_structure} to describe the coamoeba $\calA_B$
more explicitly, then study the zonotope $Z_B$ generated by $B$, before making an
important definition and giving our proof of Theorem~\ref{T:NP}.

The line $\ell_B$ is parametrized by the forms 
$z\mapsto\langle \bb_i,z\rangle$, for $i=1,\dotsc,N{+}1$.
Let $\xi_i\in\R\P^1$ be the zero of the $i$th form,
and suppose these are in  a weakly increasing cyclic order on $\R\P^1$,
\[
   \xi_1\ \leq\ \xi_2\ \leq\ \dotsb\ \leq\ \xi_{N+1}\,.
\]
Next, identify $\P^1\smallsetminus\{\xi_{N+1}\}$ with
$\C$, so that $\xi_{N+1}$ is the point $\infty$ at infinity, and suppose that the
distinct zeroes are
\[
   \zeta_1\ <\ \zeta_2\ <\ \dotsb\ <\ \zeta_M\ <\ \zeta_{M{+}1}\ =\ \infty\,.
\]

By the description of the coamoeba $\calA(\ell_B)$ of
Section~\ref{S:realline} and Lemma~\ref{L:coamoeba_structure}, we see that the coamoeba
$\calA_B$ is composed of two components, each bounded by polygonal paths that are the
images of the boundary of $\calA(\ell_B)$ under the map $\Arg(\beta)$.
For each $j=1,\dotsc,M{+}1$, set
\[
  \defcolor{\bc_j}\ :=\ \Arg(\beta)(\bbf_j)\ =\ \sum_{i\colon \langle\bb_i,\zeta_j\rangle=0} \bb_i\,.
\]
The components of $\calA_B$ correspond to the half planes of $\P^1$, and the boundary
along each is the polygonal path with edges $\pm\pi\bc_1,\dotsc,\pm\pi\bc_{M+1}$ with the
$+$ signs for the upper half plane and $-$ signs for the lower half plane.
The complete description requires the following proposition, which is explained
in~\cite[\S~2]{NP10}. 

\begin{proposition}\label{P:unrammified}
 Suppose that $M>1$.
 Then the composition
\[
   \P^1\smallsetminus\{\zeta_1,\dotsc,\zeta_{M+1}\}
          \xrightarrow{\,\Psi_B\,}D_B\xrightarrow{\,\Arg\,}\calA_B\ \subset\ \T^2
\]
 is an immersion when restricted to $\P^1\smallsetminus\R\P^1$ (in fact it is locally a
 covering map).
\end{proposition}

The edges $\pm\pi\bc_1,\dotsc,\pm\pi\bc_{M+1}$ decompose $\T^2$ into polygonal regions.
Over each polygonal region the map of Proposition~\ref{P:unrammified} has a constant number of
preimages.  
This number of preimages equals the winding number of the polygonal path around that region.
Then the pushforward $\Arg(\beta)_*(\overline{\calA(\ell_B)})$ of the coamoeba chain of 
the line $\ell_B$ is the chain in $\T^2$ where the multiplicity of a region is this number
of preimages/winding number. 
This equals the coamoeba chain of $D_B$.
We will write \defcolor{$\overline{\calA_B}$} for this chain
$\Arg(\beta)_*(\overline{\calA(\ell_B)})$, as our arguments use the pushforward.

There is another natural chain we may define from the vector configuration $B$.
Let $\defcolor{\overline{\bO,\pi\bb_i}}$ be the directed line segment in $\R^2$ connecting the
origin to the endpoint of the vector $\pi\bb_i$.
Let $\defcolor{Z_B}\subset\R^2$ be the Minkowski sum of the line segments
$\overline{\bO,\pi\bb_i}$ for $\bb_i\in B$. 
This is a centrally symmetric zonotope as $\sum_i \bb_i=\bO$.
We will also write $Z_B$ for its image in $\T^2$, considered now as a chain.
For any $\bv\in\R^2$, the points 
\[
   q\ :=\ \sum_{\langle\bb_i,\bv\rangle>0} \bb_i \qquad\mbox{and}\qquad
   q'\ :=\  \sum_{\langle\bb_i,\bv\rangle\geq0} \bb_i
\]
are vertices of $Z_B$ which are extreme in the direction of $\bv$.
These differ only if the line $\R\bv$ represents a zero $\zeta_j$ of one of the
forms, and then the edge between them is $\pi\bd_j$, where
 \begin{equation}\label{Eq:bdj}
   \bd_j\ :=\  \sum_{i\colon \langle\bb_i,\bv\rangle=0}
     \sign(\langle\bb_i,\bw\rangle)\; \bb_i,,
 \end{equation}
where $\bw$ is a vector such that $\langle -\bw,q\rangle>\langle -\bw,q'\rangle$ and
$\sign(x)\in\{\pm1\}$ is the sign of the real number $x$.
Thus $\bd_j$ is the vector parallel to any $\bb_i$ with $\langle\bb_i,\zeta_j\rangle=0$
whose length is the sum of the lengths of these vectors and its direction is such that 
$\langle\bd_j,\bw\rangle>0$.

Starting at a vertex of $Z_B$ and moving, say clockwise, the successive edge vectors will be
the vectors $\{\pm\pi\bd_1,\dotsc,\pm\pi\bd_M,\pm\pi\bd_{M+1}\}$ occuring in a cyclic
clockwise order.
This may be seen on the right in~\eqref{Eq:A-discr}, where $Z_B$ is the octagon.
Its southeastern-most vertex is $\pi\bb_1+\pi\bb_3$ (corresponding to the vector $\bv_1=-\bb_2$,
and the edges encountered from there in clockwise order are
$-\pi\bb_1,\pi\bb_2,-\pi\bb_3, \pi\bb_4, \pi\bb_1,-\pi\bb_2, \pi\bb_3,-\pi\bb_4$.
(Here, $\bd_j=\bb_j$)

Before giving our proof of Theorem~\ref{T:NP}, we make an important definition.
 Let $B=\{\bb_1,\dotsc,\bb_{N+1}\}$ be a multiset of vectors in $\Z^2$ that span
 $\R^2$ and whose sum is $\bO$.
 Write \demph{$\cone(\bb_i,\bb_j)$} for the cone generated by the vectors $\bb_i, \bb_j$.
 Suppose that $\bv$ is any vector in $\R^2$ not pointing in the direction of a vector in $B$,
 and set
 \begin{equation}\label{Eq:dBv}
   \defcolor{d_{B,\bv}}\ :=\ \sum_{\bv\in\cone(\bb_i,\bb_j)} |\bb_i\wedge \bb_j|\ .
 \end{equation}
 Here $|\bb_i\wedge \bb_j|$ is the absolute value of the determinant of the matrix whose columns are
 the two vectors, which is the area of the parallelogram generated by $\bb_i$ and $\bb_j$.

\begin{lemma}\label{L:independent}
 The sum~$\eqref{Eq:dBv}$ is independent of choice of $\bv$.
\end{lemma}

\begin{proof}
 The rays generated by elements of $B$ divide $\R^2$ into regions.
 The sum~\eqref{Eq:dBv} depends only upon the region containing $\bv$---it is a sum over all 
 cones containing the given region.
 To show its independence of region, let $\bv,\bv'$ lie in adjacent regions with
 $\defcolor{\bu}$ a vector generating the ray separating the regions.
 Suppose that the vectors in $B$ are indexed so that
 $\bb_\kappa,\bb_{\kappa+1},\dotsc,\bb_{\mu-1}$ are the vectors with direction $-\bu$ and
 $\bb_\mu,\bb_{\mu+1},\dotsc,\bb_\lambda$ are the vectors with direction $\bu$.
 Then the sums for $d_{B,\bv}$ and $d_{B,\bv'}$ both include the sum over all cones whose relative
 interior contains $\bu$, but have different terms involving cones with one generator among
 $\bb_\mu,\dotsc,\bb_\lambda$. 
 All such cones appear, and up to a sign, the difference $d_{B,\bv}-d_{B,\bv'}$ is equal to
\begin{multline*}
  \bigl(\bb_\mu +\dotsb+ \bb_\lambda\bigr) \wedge \bigl(\bb_1 +\dotsb+ \bb_{\kappa-1}\ +\
  \bb_{\lambda+1}+\dotsb+ \bb_{N+1}\bigr)\\
   \ =\  \bigl(\bb_\mu +\dotsb +\bb_\lambda\bigr) \wedge (\bb_1+\dotsb+\bb_{N+1})\ =\ 0\,,
\end{multline*}
 which proves the lemma.
\end{proof}

\begin{remark}
 The sum~\eqref{Eq:dBv} is known to coincide with the normalized volume of the convex hull of the
 vector configuration $A$ that is Gale dual to $B$ (see~\cite{DS02}), so Lemma~\ref{L:independent}
 also follows from this fact.
 We will henceforth write \demph{$d_B$} for this volume/sum.
 \hfill$\QED$
\end{remark}

\begin{example}
 Consider the sum~\eqref{Eq:dBv} for the vector configuration $B$ of
 Example~\ref{Ex:rational_cubic}.
 There are four choices for the vector $\bv$ as indicated below
 \begin{equation}\label{Eq:B-and-v}
   \raisebox{-46pt}{\begin{picture}(100,93)(-15,-3)
    \put(0,0){\includegraphics{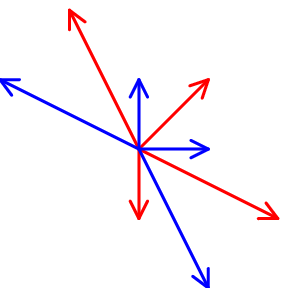}}
    \put( 36,63){$\bb_4$}   \put(62,37){$\bb_1$}
    \put(-13,57){$\bb_2$}   \put(62, -3){$\bb_3$}

    \put( 36,12){$\bv_1$}   \put(81,21){$\bv_2$}
    \put( 61,61){$\bv_3$}   \put( 7,79){$\bv_4$}

   \end{picture}}
 \end{equation}
 The vector $\bv_1$ lies only in $\cone(\bb_2,\bb_3)$, and we have
 $\bb_2\wedge \bb_3=|\begin{smallmatrix}-2&1\\1&-2\end{smallmatrix}|=3$.
 The vector $\bv_2$ lies in $\cone(\bb_3,\bb_1)$ and $\cone(\bb_3,\bb_4)$, and 
 we have $\bb_3\wedge \bb_1+\bb_3\wedge \bb_4=
  |\begin{smallmatrix}1&-2\\1&0\end{smallmatrix}|+
  |\begin{smallmatrix}1&-1\\0&1\end{smallmatrix}|=2+1=3$.
 Similarly, $\bv_3$ lies in $\cone(\bb_3,\bb_4)$, $\cone(\bb_1,\bb_4)$, and $\cone(\bb_1,\bb_2)$,
 and $\bb_3\wedge \bb_4+\bb_1\wedge \bb_4+\bb_1\wedge \bb_2=1+1+1=3$, and the calculation for
 $\bv_4$ is the mirror-image of that for $\bv_2$.
 In every case, $d_{B,\bv_i}=3$, and so $d_B=3$.
 \hfill$\QED$
\end{example}


\begin{theorem}\label{Th:Cycle}
 The sum, $\overline{\calA_B}+Z_B$, of the coamoeba chain of $D_B$ and the 
 $B$-zonotope chain is a cycle in $\T^2$ which equals $d_B[\T^2]$.
\end{theorem}

\begin{proof}
 We will show that $\Arg(\beta)_*[Z(\ell_B)]=[Z_B]$, which implies that 
\[
   [\overline{\calA_B}+Z_B]\ =\ \Arg(\beta)_*[\overline{\calA(\ell_B)} + Z(\ell_B)]
\]
 is a cycle, as $\Arg(\beta)_*[\overline{\calA(\ell_B)}]=[\overline{\calA_B}]$.
 Since  $\Arg(\beta)_*(\be_i\wedge\be_j)=\bb_i\wedge\bb_j\cdot[\T^2]$, the formula of
 Theorem~\ref{Th:homology} will give us the homology class of $[\overline{\calA_B}+Z_B]$.
 We will use~\eqref{Eq:dBv} and Lemma~\ref{L:independent} to show that it equals
 $d_B[\T^2]$.
 This will imply the theorem as we will show that there is an ordering
 of the vectors $B$ such that the map $\Arg(\beta)\colon Z(\ell_B)\to Z_B$ in the
 universal covers $\R^N\to\R^2$ is injective.  

 Recall that $\xi_1,\dotsc,\xi_{N+1}$ are points of $\R\P^1$ with $\langle \bb_i,\xi_i\rangle=0$
 and $\zeta_1,\dotsc,\zeta_{M+1}$ are the distinct points among them.
 Let $\bO\neq\defcolor{\bv}\in\R^2$ represent $\xi_{N+1}=\zeta_{M+1}$ 
 (so that $\langle\bb_{N+1},\bv\rangle=0$)
 and choose $x\in\R^2$ to be a point with $\langle\bb_{N+1},x\rangle=1$.
 Then $t\mapsto x+t\bv$ gives a parametrization of $\R\P^1$ with $\infty=\zeta_{M+1}$, and 
 identifies $\R$ with $\R\P^1\smallsetminus\{\infty\}$.
 
 To agree with Definition~\ref{D:conventions}, we suppose that the points
 of $B$ are ordered so that~\eqref{Eq:A} and~\eqref{Eq:B} hold.
 Thus there are integers
 $1=m_1<\dotsb<m_{M+1}<m_{M+2}=N{+}2$ such that
\[
   \langle \bb_i,\zeta_j\rangle=0\ 
   \Longleftrightarrow\ m_j\leq i<m_{j+1}\,.
\]
 We further suppose that $B$ is ordered so that one of~\eqref{Eq:inc}
 or~\eqref{Eq:dec} holds for every $j=1,\dotsc,M{+}1$.
 Specifically, let $\bw:=x+\tau \bv$ for some fixed $\tau<\zeta_1$.
 Then there exist integers $n_1,\dotsc,n_{M+1}$ such that for each $j=1,\dotsc,M{+}1$
 we have $m_j< n_j\leq m_{j+1}$ and either
 \begin{eqnarray*}
  \langle\bb_{m_j},\bw\rangle\,,\ \dotsc\,,\ \langle\bb_{n_j-1},\bw\rangle
   &<\ 0\ <&
  \langle\bb_{n_j},\bw\rangle\,,\ \dotsc\,,\ \langle\bb_{m_{j+1}-1},\bw\rangle\,,
    \makebox[.1in][l]{\qquad or}\\
  \langle\bb_{m_j},\bw\rangle\,,\ \dotsc\,,\ \langle\bb_{n_j-1},\bw\rangle
   &>\ 0\ >&
  \langle\bb_{n_j},\bw\rangle\,,\ \dotsc\,,\ \langle\bb_{m_{j+1}-1},\bw\rangle\,.
 \end{eqnarray*}
 For $i=1,\dotsc,N{+}1$, let $\sgn_i\in\{\pm 1\}$ be the sign of $\langle\bb_i,\bw\rangle$.
 Note that $\sgn_{N+1}=1$.

 Define $\bbf_j,\bg_j,\bh_j$ as in Definition~\ref{D:conventions},
\[
  \bbf_j\ :=\ \sum_{i=m_j}^{m_{j+1}-1} \be_i\,,\qquad
  \bg_j\  :=\ \sum_{i=m_j}^{n_{j}-1} \be_i\,,\qquad\mbox{and}\qquad
  \bh_j\  :=\ \sum_{i=m_j}^{m_{j+1}-1} \sgn_i\be_i\,.
\]

 Consider now the following affine parametrization of $\ell_B\subset\P^N$,
\[
  \Phi_B\ \colon\ t\ \longmapsto\ 
   [\langle\bb_1,x+t\bv\rangle\ \colon 
    \dotsb\ \colon\
     \langle\bb_N,x+t\bv\rangle\ \colon\ \langle\bb_{N+1},x+t\bv\rangle=1]\,.
\]
 Let $\tp_1\in\{0,\pi\}^N\in\R^N$ be the point whose $i$th coordinate is 
  $\arg(\sgn_i)$.
 Its image $p_1\in\T^N$ is the point on the coamoeba of $\ell_B$ coming from 
 the real points $\Phi_B(-\infty,\zeta_1)$.

 We describe $\Arg(\beta)(Z(\ell_B))$ in the universal cover $\R^2$ of $\T^2$.
 For each $j=1,\dotsc,2M{+}2$, set $\defcolor{\tq_j}:=\Arg(\beta)(\tp_j)$ and
 $\defcolor{\tqp_j}:=\Arg(\beta)(\tpp_j)$.
 Since 
 \begin{equation}\label{Eq:def_p1tilde}
    \tp_{1,i}\ =\ 
    \left\{ \begin{array}{rcl} \pi &&\mbox{if }\langle\bb_i,\bw\rangle<0\\
              0 &&\mbox{if }\langle\bb_i,\bw\rangle>0
    \end{array}\right.\ ,
 \end{equation}
 we have 
\[
   \tq_1\ =\ \pi\cdot\sum_{\langle\bb_i,\bw\rangle<0}\bb_i\ ,
\]
 and so $\tq_1$ is a vertex of $Z_B$ which is extreme in the direction of $-\bw$.

 The zonotope chain $Z(\ell_B)$ is a union of the triangles 
 \begin{equation}\label{Eq:ZC_triangles}
   \conv(\bO,\tp_{j+1},\tpp_j)
    \qquad\mbox{and}\qquad
   \conv(\bO,\tpp_j,\tp_j)
    \qquad\mbox{for}\ j=M{+}2,\dotsc,1\,,
 \end{equation}
 where the second is degenerate if $\tp_j=\tpp_j$.
 Thus $\Arg(\beta)(Z(\ell_B))$ will be the union of the (possibly degenerate) triangles 
 \begin{equation}\label{Eq:Zone_triangles}
   \conv(\bO,\tq_{j+1},\tqp_j)
    \qquad\mbox{and}\qquad
   \conv(\bO,\tqp_j,\tq_j)
    \qquad\mbox{for}\ j=M{+}2,\dotsc,1\,,
 \end{equation}
 For $j\leq M{+}1$, $\tp_{j+1}=\tp_j+\pi\bh_j$, so
\[
   \tq_{j+1}\ =\ \tq_j+\pi\Arg(\beta)(\bh_j)\ =\ 
    \tq+\pi\bd_j\,,
\]
 which we see by~\eqref{Eq:bdj} (with the vector $\bw=x+\tau\bv$) and our definition of
 $\sgn_i$.
 If we fix the orientation so that $\bv$ is clockwise of $\bb_{N+1}$, then 
 by our choice of ordering of the zeroes $\zeta_j$, the lines $\R\bd_1, \dotsc,\R\bd_{M+1}$
 occur in clockwise order.
 Since $\langle\bd_j,\bw\rangle>0$ and $\tq_1$ is extreme in the direction of $-\bw$, 
 the vectors $\pi\bd_1,\dotsc,\pi\bd_{M{+}1}$ will form the edges of the zonotope 
 starting at $\tq_1$ and moving clockwise.
 It follows from the discussion following~\eqref{Eq:bdj} that $\tq_1,\dotsc,\tq_{2M+2}$ form
 the vertices of the zonotope $Z_B$. 
 This implies that no $\tq_j$ coincides with the origin $\bO$.

 All that remains is to understand the two triangles~\eqref{Eq:Zone_triangles}
 for those $j$ when $\tqp_j\neq\tq_j$.
 In this case, $\tpp_j=\tp_j+2\pi\sgn_{m_j}\bg_j$, and so 
\[
  \tqp_j\ =\ \tp_j\ +\  2\pi\sgn_{m_j}\sum_{i=m_j}^{n_j-1} \bb_i\ =\ 
   \tp_j\ +\  2\pi\sum_{i=m_j}^{n_j-1} \sgn_i\bb_i\,.
\]
 Since $\bb_{m_j},\dotsc,\bb_{m_{j+1}-1}$ are parallel, $\tq_j,\tqp_j$, and
 $\tq_{j+1}$ are collinear.
 This implies that 
\[
   \Arg(\beta)_* [ \conv(\bO,\tpp_j,\tp_j) + \conv(\bO,\tp_{j+1},\tpp_j)]\ =\ 
   [\conv(\bO,\tq_{j+1},\tq_j)]\,,
\]
 which shows that $\Arg(\beta)_*[Z(\ell_B)]=[Z_B]$.

 Indeed, if $\tqp_j$ lies between $\tq_j$ and $\tq_{j+1}$ then $\Arg(\beta)$ 
 preserves the orientation of the triangles~\eqref{Eq:ZC_triangles} and is therefore
 injective over their images, whose union is $\conv(\bO,\tq_{j+1},\tq_j)$.
 Otherwise, the two triangles~\eqref{Eq:Zone_triangles} have opposite orientations and 
\[
    \conv(\bO,\tqp_j,\tq_j)\ \supset\ \conv(\bO,\tq_{j+1},\tqp_j)\,,
\]
 so that $\Arg(\beta)_* [ \conv(\bO,\tpp_j,\tp_j) + \conv(\bO,\tp_{j+1},\tpp_j)]$
 equals
\[
  [ \conv(\bO,\tqp_j,\tq_j)]-[ \conv(\bO,\tq_{j+1},\tqp_j)]
  \ =\ [\conv(\bO,\tq_{j+1},\tq_j)]\,.
\]

 Theorem~\ref{Th:homology}, Equation~\eqref{Eq:def_p1tilde}, and 
 $\Arg(\beta)_*(\be_i\wedge \be_j)=\bb_i\wedge \bb_j\cdot[\T^2]$, show that 
 \[
   \Arg(\beta)_*[\overline{\calA(\ell_B)}+Z(\ell_B)]\ =\  [\T^2]\cdot
   \sum_{\substack{1\leq i<j\leq N\\\langle\bb_i,\bw\rangle>0>\langle\bb_j,\bw\rangle}} 
     \bb_i\wedge \bb_j\,.
 \]
 We will show that this equals $d_B[\T^2]$.
 Observe that if $\bb_i$ and $\bb_j$ are parallel, then $\bb_i\wedge\bb_j=0$ and they do not
 contribute to the sum.
 We will consider the sum with the restriction that the vectors $\bb_i$ and $\bb_j$ are not parallel.

 Set $\defcolor{\bw^\perp}:=-\bb_{N+1}+\bw/\langle\bw,\bw\rangle$, which is orthogonal to
 $\bw$.
 Suppose that $\bv$ is clockwise of $\bb_{N+1}$, as below.
\[
  \begin{picture}(93,104)(-13,-20)
   \put(-10,-10){\includegraphics{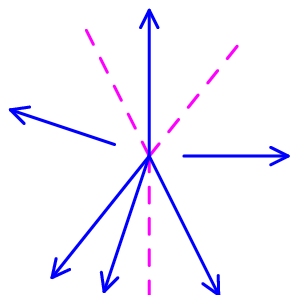}}
   \put(-1,-12){$\bb_i$} \put(58,-13){$\bb_j$}
   \put(25,78){$\bb_{N+1}$}  \put(61,33){$\bv$}
   \put(-13,32){$\bw$}   \put(16,-18){$\bw^\perp$}
  \end{picture}
\]
 By our choice of $\bw$, the lines $\R\bw^\perp,\R\bb_1,\dotsc,\R\bb_{N+1}$ occur in weak
 clockwise order with $\R\bw^\perp$ distinct from the rest.
 Suppose now that $1\leq i<j\leq N$ where 
 \begin{equation}\label{Eq:condition}
    \langle\bb_i,\bw\rangle\ >\ 0\ >\ \langle\bb_j,\bw\rangle\,,
 \end{equation}
 and $\bb_i$ and $\bb_j$ are not parallel.
 The cone spanned by $\bb_i$ and $\bb_j$ meets a half ray of $\R\bw^\perp$, with $\bb_i$
 to the left of $\R\bw^\perp$ and $\bb_j$ to the right of $\R\bw^\perp$, by~\eqref{Eq:condition}.
 Since $\R\bw^\perp,\R\bb_i$, and $\R\bb_j$ occur in clockwise order, we must have that 
 $\bw^\perp\in\cone(\bb_i,\bb_j)$, which shows that 
 \[
   \sum_{\substack{1\leq i<j\leq N\\\langle\bb_i,\bw\rangle<0<\langle\bb_j,\bw\rangle}} 
     \bb_i\wedge \bb_j
   \ =\ 
   \sum_{\substack{1\leq i<j\leq N\\\bw^\perp\in\cone(\bb_i,\bb_j)}} 
     \bb_i\wedge \bb_j\ =\ d_{B,\bw^\perp}\ =\ d_B\,.
 \]
 The sum equals $d_{B,\bw^\perp}$ because if $\bb_j$ is counter clockwise from $\bb_i$
 by~\eqref{Eq:condition} and the condition that $\bw^\perp\in\cone(\bb_i,\bb_j)$ with
 $i<j$.
 Thus $\bb_i\wedge\bb_j>0$.

 We complete the proof by noting that $\tqp_j$ will lie between $\tq_j$ and $\tq_{j+1}$
 if either $n_j=m_{j+1}$, so that $\bg_j=\bbf_j$, or if 
\[
   \|\bg_j\|\ =\ \|\sum_{i=m_j}^{n_j-1}\bb_j\|
  \ =\ \sum_{i=m_j}^{n_j-1}\|\bb_j\|\ \leq\ 
   \sum_{i=n_j}^{m_{j+1}-1}\|\bb_j\|\ =\ \|\bbf_j-\bg_j\|\,,
\]
 as $\bb_{m_j},\dotsc,\bb_{n_j-1}$ have the same direction which is opposite to the (common)
 direction of $\bb_{n_j},\dotsc,\bb_{m_{j+1}-1}$.
 If this does not occur for our given order, then we simply reverse the vectors
 $\bb_{m_j},\dotsc,\bb_{m_{j+1}-1}$, replacing $\bg_j$ with $\bbf_j-\bg_j$.
\end{proof}

\begin{example}\label{ex:parallel}
 The last point in the proof about the injectivity of 
\[
   \Arg(\beta)\ \colon\ Z(\ell_B)\ \longrightarrow\ Z_B
\]
 (and more generally the arguments when $B$ has parallel vectors)
 is geometrically subtle.
 We expose this subtlety in the following two examples.
 Suppose that $B$ consists of the vectors $(1,0),(0,1),(-2,-2)$, and $(1,1)$, 
\[
    \includegraphics{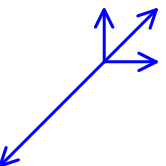}
\]
 When $\bv=(1,-1)$ and $x=(\frac{1}{2},\frac{1}{2})$, then $\ell_B$ has the parametrization 
 \begin{equation}\label{Eq:first_P}
   z\ \longmapsto\ [ \tfrac{1}{2}+z\;:\; \tfrac{1}{2}-z\;:\; -2\;:\; 1]\,,
 \end{equation}
 which is the second line in our running Examples~\ref{Ex:CoArepeat},~\ref{Ex:CoZrepeat},
 and~\ref{Ex:homology_class}. 
 In this case the image $\Arg(\beta)(Z(\ell_B))$ is shown on the left of Figure~\ref{F:images}.
 It is superimposed over a fundamental domain and dashed lines $\theta_1,\theta_2=n\pi$ for
 $n\in\Z$. 
 The segments $\tq_3,\tqp_2$ and $\tq_6,\tqp_5$ are covered in both directions as
 $\Arg(\beta)(P(\ell_B))$ backtracks over these segments.
 In fact, the triangles
\[
   \conv(\bO,\tq_3,\tqp_2)\qquad\mbox{and}\qquad
   \conv(\bO,\tq_6,\tqp_5)
\]
 have orientation opposite of the other triangles.
 The medium shaded parts (near $\tqp_2$ and $\tqp_5$) are covered twice and the darker shaded
 parts near $\bO$ are covered thrice.
\begin{figure}[htb]

 \begin{picture}(177,126)(-11,-11)
  \put(0,0){\includegraphics{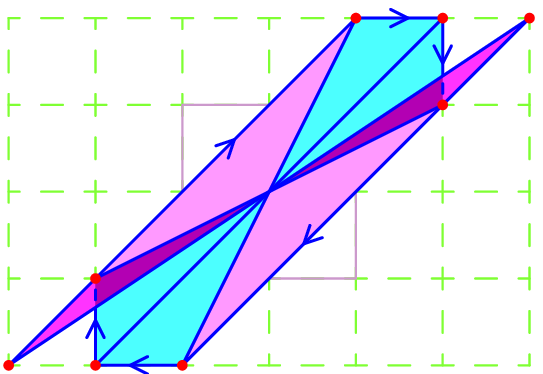}}

  \put( 71, 52){$\bO$}
  \put( 23,-10){$\tq_1$}  \put( 48,-10){$\tq_2$}
  \put(156,102){$\tqp_2$} \put(129, 67){$\tq_3$}
  \put(129,107){$\tq_4$}  \put( 93,107){$\tq_5$}
  \put(-11, -3){$\tqp_5$} \put( 16, 32){$\tq_6$}

 \end{picture}
 \qquad
 \begin{picture}(177,126)(-11,-11)
  \put(0,0){\includegraphics{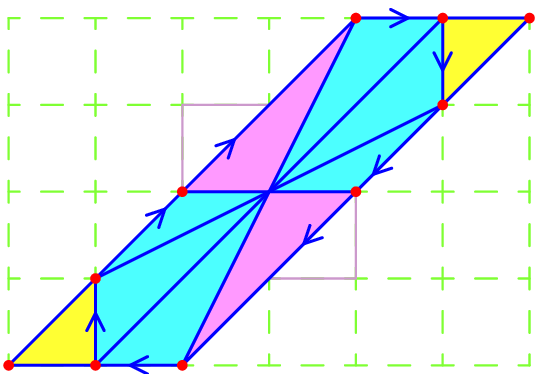}}

  \put( 71, 54){$\bO$}
  \put( 23,-10){$\tq_1$}  \put( 48,-10){$\tq_2$}
  \put(104, 42){$\tqp_2$} \put(129, 67){$\tq_3$}
  \put(129,107){$\tq_4$}  \put( 93,107){$\tq_5$}
  \put( 39, 57){$\tqp_5$} \put( 16, 32){$\tq_6$}
  \thicklines
  \put(130,11){\White{\vector(-1,0){114}}} \put(130,11){\White{\vector(-1,0){115}}}
  \put(138,21){\White{\vector(0,1){78}}}  \put(138,22){\White{\vector(0,1){78}}}

  \thinlines
  \put(130,11){\vector(-1,0){113}}
  \put(138,21){\vector(0,1){77}}
  \put(132,9){$\calA_B$}
 \end{picture}

 \caption{Images of $\Arg(\beta)(Z(\ell_B))$}
 \label{F:images}
\end{figure}

Now suppose that the vectors in $B$ are in the order $(1,0)$, $(1,1)$, $(-2,-2)$, and $(0,1)$,
and $v=(-1,0)$ and $x=(0,1)$.
Then $\ell_B$ is parametrized by
\[
   z\ \longmapsto\ [-z\;:\; 1-z \;:\; 2z-2 \;:\; 1]\,,
\]
In this case the image $\Arg(\beta)(Z(\ell_B))$ is equal to the zonotope
$Z_B$, and is shown on the right of Figure~\ref{F:images}, together with the
coamoeba $\calA_B$. 
As explained in the proof of Theorem~\ref{Th:Cycle}, the image equals the zonotope because
in the pair of parallel vectors $(1,1)$ and $(-2,-2)$, the shorter comes first in this
case, while in the previous case, the shorter one came second.

In both cases (which are just different parametrizations of the same line)
$\Arg(\beta)_*[Z(\ell_B)]=[Z_B]$ as shown in the proof of
Theorem~\ref{Th:Cycle}, and the coamoebas coincide.
Furthermore, $[\overline{\calA_B}+Z_B]=2[\T^2]$ for both, as $d_B=2$.
\hfill\QED
\end{example}

\def\cprime{$'$}
\providecommand{\bysame}{\leavevmode\hbox to3em{\hrulefill}\thinspace}
\providecommand{\MR}{\relax\ifhmode\unskip\space\fi MR }
\providecommand{\MRhref}[2]{%
  \href{http://www.ams.org/mathscinet-getitem?mr=#1}{#2}
}
\providecommand{\href}[2]{#2}


\begin{thebibliography}{10}

\bibitem{appell}  P. Appell, 
 \emph{Sur les s\'eries hyperg\'eometriques de deux variables et sur des
  \'equations diff\'erentielles lin\'eaires aux d\'eriv\'ees partielles}, 
  Comptes Rendus Hebdomadaires des S\'eances de l'Acad\'emie des Sciences. S\'eries A et B,
 \textbf{90}, (1880), 296--298.

\bibitem{Beukers}
F.~Beukers, \emph{Monodromy of {$A$}-hypergeometric functions}, {\tt
  arXiv.org/1101.0493}.

\bibitem{DS02}
A.~Dickenstein and B.~Sturmfels, \emph{Elimination theory in codimension
  2}, J. Symbolic Comput. \textbf{34} (2002), no.~2, 119--135.

\bibitem{GKZ89}
I.~M. Gel{\cprime}fand, M.~M. Kapranov, and A.~V. Zelevinsky,
  \emph{Hypergeometric functions and toric varieties}, Funktsional. Anal. i
  Prilozhen. \textbf{23} (1989), no.~2, 12--26.

\bibitem{GKZ94}
\bysame, \emph{Discriminants, resultants, and multidimensional determinants},
  Mathematics: Theory \& Applications, Birkh\"auser Boston Inc., Boston, MA,
  1994.

\bibitem{Johansson}
P.~Johansson, \emph{The argument cycle and the coamoeba}, Complex Variables and Elliptic
Equations, DOI: 10.1080/17476933.2011.592581, 2011.

\bibitem{K91}
M.~M. Kapranov, \emph{A characterization of {$A$}-discriminantal hypersurfaces
  in terms of the logarithmic {G}auss map}, Math. Ann. \textbf{290} (1991),
  no.~2, 277--285.

\bibitem{lauricella} G. Lauricella, 
  \emph{Sulla funzioni ipergeometriche a pi\`u variabili}, 
 Rend. Circ. Math. Palermo \textbf{7} (1893), 111--158.

\bibitem{Nilsson}
L.~Nilsson, \emph{Amoebas, discriminants, and hypergeometric functions},
  Ph.D. thesis, Stockholm University, 2009.

\bibitem{NP10}
L.~Nilsson and M.~Passare, \emph{Discriminant coamoebas in dimension
  two}, J. Commut. Algebra \textbf{2} (2010), no.~4, 447--471.

\bibitem{NS}
M.~Nisse and F.~Sottile, \emph{The phase limit set of a variety}, {\tt
  arXiv:1106.0096}, Algebra and Number Theory, to appear.

\bibitem{SST}
M.~Saito, B.~Sturmfels, and N.~Takayama, \emph{Gr\"obner
  deformations of hypergeometric differential equations}, Algorithms and
  Computation in Mathematics, vol.~6, Springer-Verlag, Berlin, 2000.

\end{thebibliography}
\end{document}